\documentclass[12pt,a4paper]{article}

\usepackage{amsmath}    
\usepackage{amssymb}
\usepackage{graphicx}   
\usepackage{verbatim}   
\usepackage{color}      
\usepackage{hyperref}   
\usepackage{enumerate}
\usepackage{listings}
\usepackage{xcolor}
\usepackage{slashed}
\usepackage{amsthm}
\usepackage{amsmath}    
\usepackage{amssymb}
\usepackage{mathrsfs}
\usepackage{bbold}

\usepackage[final]{showlabels}

\usepackage{cleveref}
\crefformat{footnote}{#2\footnotemark[#1]#3}

\newcommand{\del}{\partial}
\renewcommand{\theta}{\vartheta}
\renewcommand{\phi}{\varphi}

\newcommand{\dd}{\mathrm{d}}

\renewcommand{\and}{\wedge}

\newcommand{\sgn}{\mathrm{sign\,}}

\definecolor{mygreen}{rgb}{0,0.6,0}
\definecolor{mygray}{rgb}{0.5,0.5,0.5}
\definecolor{mymauve}{rgb}{0.58,0,0.82}

\newtheorem{lemma}{Lemma}
\newtheorem{theorem}{Theorem}

\newtheorem{example}{Example}
\newtheorem{definition}{Definition}

\renewcommand{\title}{A TVD Active Flux method for scalar conservation laws}

\newcommand{\authorOne}{Wasilij Barsukow
\footnote{CNRS, Institut de Mathématiques de Bordeaux (IMB), UMR 5251, 351 Cours de la Libération, 33405 Talence, France, \url{wasilij.barsukow@math.u-bordeaux.fr}, \emph{Present address:} Imperial College London, IRL2004, Huxley Building, South Kensington Campus, London, United Kindgom}
}

\begin{document}

\begin{center} \Large
\title

\vspace{1cm}

\date{}
\normalsize

\authorOne
\end{center}

\begin{abstract}

Focusing on scalar conservation laws, a TVD theory is developed for the fully discrete, single-stage Active Flux method. The classical choice of degrees of freedom is considered, yielding a third-order method. Starting from an adaptation of the concept of (discrete) total variation to Active Flux, a corresponding adaptation of Harten's lemma is achieved. Finally, the high-order method is endowed with parameters chosen such that they can be efficiently computed to fulfill the conditions of the lemma.



\end{abstract}

\begin{center}
{\footnotesize Dedicated to the memory of Philip L. Roe (1938--2026).}
\end{center}

\newcommand{\cfl}{\nu}

\section{Introduction}

Active Flux for scalar conservation laws
\begin{align}
\del_t q + \del_x f(q) &= 0 \qquad q \colon \mathbb R^+_0 \times \mathbb R \to \mathbb R, f \colon \mathbb R \to \mathbb R \label{eq:conslaw}
\end{align}
has been introduced in \cite{vanleer77,eymann13}.
Its high order of accuracy is achieved by incorporating and evolving several degrees of freedom per computational cell $C_i := [x_{i-\frac12},x_{i+\frac12}]$ (with $\Delta x := x_{i+\frac12}  - x_{i-\frac12}$ constant for simplicity). The classical choice is third order of accuracy with point values $q_{i+\frac12}^n \simeq q(t^n, x_{i+\frac12})$ at cell interfaces and cell averages $\bar q_i^n \simeq \frac{1}{\Delta x} \int_{x_{i-\frac12}}^{x_{i+\frac12}} q(t^n, x) \, \dd x$. The former are shared, such that in each cell an interpolating parabola $q^n_{\text{recon},i} \colon [-\frac{\Delta x}{2}, \frac{\Delta x}{2}] \to \mathbb R$ can be constructed from the knowledge of the average and the two point values at its endpoints:
\begin{align}
 q^n_{\text{recon},i}(x) &= -3 (2 \bar q_i^n - q_{i-\frac12}^n - q_{i+\frac12}^n) \frac{x^2}{\Delta x^2} + (q_{i+\frac12}^n - q_{i-\frac12}^n) \frac{x}{\Delta x} + \frac{6 \bar q_i^n - q_{i-\frac12}^n - q_{i+\frac12}^n}{4}  \label{eq:parabolicrecon}
\end{align}

This parabola (usually called \emph{reconstruction}) joins continuously with the parabolas in the neighbouring cells giving rise to a globally continuous approximation/reconstruction
\begin{align}
q^n_{\text{recon}}(x) =  q^n_{\text{recon},i}(x - x_i) \quad \text{if } x \in C_i 
\end{align}
where $x_i = \frac12(x_{i+\frac12} + x_{i-\frac12})$ (see Figure \ref{fig:reconstructions}, center).
Thus, the approximation space of classical Active Flux is essentially that of standard $P^2$ Finite Elements. Extensions to higher polynomial degrees have been considered in \cite{abgrall22,abgrall22proceeding}, but here the classical choice is adopted.

The way the degrees of freedom of Active Flux are evolved in time is different from Finite Element methods. With details postponed to Section \ref{sec:afreview}, here it is enough to mention that Active Flux is a single-stage method. The point values are evolved in time first, and their old ($q_{i+\frac12}^n$), new ($q_{i+\frac12}^{n+1}$) and possibly intermediate values are used as quadrature points in time to approximate the time-average of the flux over the time step. This flux is then used to update the cell average like in Finite Volume methods. Since the data are continuous, Riemann solvers are not necessary (they would all simply give the physical flux by consistency, anyway, since there are no jumps in the reconstruction).  The classical choice is that of Simpson's rule which requires a point value evaluated at half the time step. This intermediate point evaluation $q_{i+\frac12}^{n+\frac12}$ will become pivotal later for achieving the TVD property.

Active Flux (being high-order and based on parabolic reconstructions) cannot be expected to be free of spurious oscillations. Provable properties for Active Flux have so far been achieved for bound preservation in \cite{chudzik23,duan24,abgrall25bound,kiechle25,barsukow25afos}, e.g. in order to guarantee positivity of the solution. This approach can in principle also be used to avoid spurious oscillations for scalar conservation laws since the latter enjoy a maximum principle, if one is able to obtain accurate \emph{local} bounds. However, these approaches give up continuity and use a Riemann solver as a fallback low-order method, which does not comply with the defining principles of Active Flux.
To prevent spurious oscillations, high-order Finite Volume methods replace the high-degree polynomials by constant reconstructions in regions where spurious oscillations might appear (in particular near discontinuities), see e.g. \cite{godlewski91,leveque02} for an introduction. This \emph{slope-limiting} is impossible for Active Flux without violating continuity (\cite{kerkmann18}): A constant function cannot simultaneously interpolate the given point values and the given average (see Figure \ref{fig:reconstructions}, left).

It is clear, though, that as long as the value $\bar q_i^n$ is contained strictly between $q_{i-\frac12}^n$ and $q_{i+\frac12}^n$, there exist reconstructions that are both monotone and continuous. Such discrete data will be referred to as strictly monotone. In fact, sometimes even the parabolic reconstruction will be oscillation-free (in the case when it degenerates to a straight line, for instance, and generally when $\bar q_i^n$ is close to $\frac{q_{i-\frac12}^n + q_{i+\frac12}^n}{2}$).

Whenever the data are strictly monotone, but the parabolic reconstruction is not, several alternative reconstruction approaches exist in the literature. Suggestions include replacing the parabolic reconstruction by piecewise parabolas (\cite{roe15,kerkmann18,barsukow20swaf}), hyperbolic reconstruction (\cite{kerkmann18}) or a power law (\cite{barsukow19activeflux}, see Figure \ref{fig:reconstructions}, right), all the while ensuring that the new reconstruction passes through the given point values and has the correct cell average. The piecewise definitions are somewhat cumbersome, while the hyperbolic reconstruction has the disadvantage of its parameters not being determinable analytically. 

\begin{figure}
 \centering
 \includegraphics[width=\textwidth]{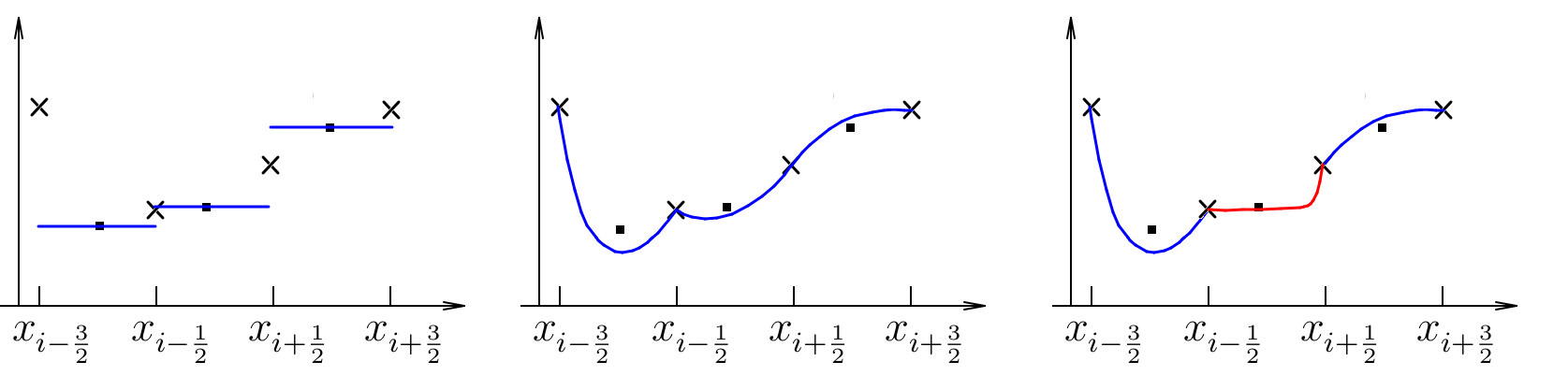}
 \caption{\emph{Left}: Constant reconstructions cannot be used for limiting while maintaining continuity. \emph{Center}: Parabolic reconstructions are globally continuous (but not continuously differentiable). They are not necessarily monotone even if the data are, see cell $C_i$. \emph{Right}: The reconstruction in cell $C_i$ has been replaced by a power law and now is monotone. No continuous reconstruction in cell $C_{i-1}$ can be monotone.}
 \label{fig:reconstructions}
\end{figure}

A continuous and monotone reconstruction becomes impossible once $\bar q_i \geq \max(q_{i-\frac12}^n,q_{i+\frac12}^n)$, or $\bar q_i \leq \min(q_{i-\frac12}^n,q_{i+\frac12}^n)$, apart from the case when all three agree (see Figure \ref{fig:reconstructions}, right). Either continuity or monotonicity has now to be abandoned, and it makes sense to abandon the latter since these situations seem to indicate that indeed the cell under consideration contains a local maximum/minimum of the data. The prescription in \cite{barsukow19activeflux}, for instance, is to reconstruct monotone discrete data by a monotone function (which is a parabola if possible, and otherwise a power law), and to continue using the parabola whenever the discrete data are themselves not monotone. The replacement of a non-monotone reconstruction by a monotone one whenever the discrete data are monotone will be called ``monotonizing the reconstruction''. 

The monotonization of the reconstruction reduces the amount of oscillations (see Section \ref{sec:numerical} for examples), but it does not remove them entirely. This is because the update of the average remains unchanged, while the quadrature used to compute the flux clearly bears the imprint of high-degree polynomials. Even if the values at the quadrature points are monotone, one faces the same problem as before, but now in time instead of space, if some of the values are close to each oher. It these situations the quadrature of the flux over/underestimates the correct value and leads to a wrong value of the average in the next time step. This can create an artifical extremum \emph{in the discrete data}, which then cannot be repaired any more by monotonizing the reconstruction. The path towards a method free of spurious oscillations therefore is to additionally modify the way how the flux is computed, and the challenge is to do so without violating continuity, i.e. without replacing the flux by some low-order Riemann solver, since continuity and the absence of Riemann solvers is a defining feature of Active Flux. A simple strategy would be to base the quadrature on the power law, or whatever function was used to monotonize the reconstruction. This, however, is very difficult to achieve for nonlinear PDEs, and would also introduce upwinding into the quadrature and thus into the average update, something untypical of Active Flux, too.

Instead, first, the notion of the discrete data being monotone is formalized by defining a total variation for the degrees of freedom of Active Flux. 
Monotone data mean data satisfying
\begin{align}
 q_{i+\frac12}^n &\leq \bar q_i^n \leq q_{i-\frac12}^n \qquad \forall i &&\text{or} & q_{i+\frac12}^n &\geq \bar q_i^n \geq q_{i-\frac12}^n \qquad \forall i \label{eq:monoton}
\end{align}
The projection of a monotone function onto these degrees of freedom, for instance, will give rise to monotone data.
It thus is natural to make the following

\begin{definition} \label{def:tv}
 The total variation (TV) associated with the degrees of freedom of Active Flux is
 \begin{align}
  \mathrm{TV}(q^n) := \sum_i |q_{i+\frac12}^n- \bar q_i^n | + \sum_i | \bar q_i^n - q_{i-\frac12}^n| \label{eq:tvdef}
 \end{align}

\end{definition}
Although the two sums look similar, they are not the same, since one can easily find data for which the first sum vanishes while the second does not.
\eqref{eq:tvdef} is the TV of a piecewise linear continuous function passing, in each cell, through the three values $q_{i-\frac12}^n, \bar q_i^n,q_{i+\frac12}^n$ (which would yield a non-conservative reconstruction!). One might conceive other definitions of the discrete total variation, which will then possibly lead to other results. From the above discussion it is clear, though, that e.g. defining the TV only on the averages as is done e.g. in \cite{xu26} for DG, is much too weak in this context. In this work, \eqref{eq:tvdef} is used throughout.

Definition \ref{def:tv} is inspired by that of total variation for Finite Difference or Finite Volume methods. Harten's lemma (\cite{harten83lemma}) then gives sufficient conditions on the numerical method to be Total-Variation Non-Increasing (TVNI). Following common usage, TVNI will be confounded with Total-Variation Diminishing (TVD) in this work. TVD Finite Volume methods arise upon introducing suitable parameters in the high-order numerical flux (which allow it to be understood as a blend with a low-order flux) and choosing their values such that the conditions of Harten's lemma are fulfilled (see e.g. \cite{sweby84}, \emph{flux limiting}). The existence of this beautiful theory is of great advantage because modifications of a flux affect two cells at once, and finding the right modification by trial and error seems virtually impossible. The plan of this paper is to repeat this programme for Active Flux.

First, a generalization of Harten's lemma for the TV \eqref{eq:tvdef} is derived that is adapted to the Active Flux method. It yields sufficient conditions for Active Flux to be TVD. The first challenge is to find the ``right'' form of the Lemma. Since the conditions are only sufficient, whenever they are not satisfied, nothing is known. There exists a multitude of lemmas of Harten's type, and ``the one'' is chosen by its ability to give \emph{useful} sufficient conditions. They should in particular be fulfilled for linear advection whenever the discrete data of Active Flux are such that the parabolic reconstruction is monotone. In this case it is easy to see that neither the point value update nor the average update can create artificial extrema, and a useful generalization of Harten's lemma should allow to prove this fact immediately.

The interesting situation is when the discrete data are monotone, but their non-monotone parabolic reconstruction has been replaced by some other continuous function. At this point also the quadrature used to determine the flux needs to be changed in order to obtain the TVD property. As is shown in Section \ref{sec:examples} with the help of several examples this is again not a simple task that could be achieved by trial and error. In particular, a naive replacement by a Riemann solver would not actually be the right solution. Here, a TVD Active Flux method can be based on the generalization of Harten's lemma once suitable free parameters have been introduced. A useful choice of a parameter is the second challenge: one needs to be able to efficiently compute it to fulfill the conditions of Harten's lemma. The idea of the present paper is that this free parameter should be the intermediate flux $f(q_{i+\frac12}^{n+\frac12})$ at half the time step, with the precise definition given in \eqref{eq:auxflux}. Out of the three quadrature points in time it is the only one expendable, since one cannot change the past, and one is hesitant to change the future, since the point value $q_{i+\frac12}^{n+1}$ is assumed to come from a monotonized reconstruction and can be expected not to violate monotonicity. The TVD Active Flux method is achieved by replacing the intermediate flux evaluation $f(q_{i+\frac12}^{n+\frac12})$ by a value determined differently.

The paper is organized as follows.
After a brief review (in Section \ref{sec:afreview}) of the Active Flux method, Harten's lemma for Active Flux is introduced in Section \ref{sec:harten}. Some intuition for the challenges associated to deriving a TVD Active Flux method is built in Section \ref{sec:examples} before such a method is achieved in Section \ref{sec:TVNI}. Numerical examples follow in Section \ref{sec:numerical}.

\section{A review of Active Flux} \label{sec:afreview}
\subsection{Evolution of the degrees of freedom}

For the update of the point values, the original suggestion from \cite{vanleer77,eymann13} is to use evolution operators: exact or approximate solutions of the initial value problem (IVP) consisting of the PDE \eqref{eq:conslaw} and the reconstruction as initial data. If the evolution operator is approximate, it needs to be of sufficiently high order of accuracy. Denote it by 
\begin{align}
E & \colon \mathbb R^+_0 \times \mathbb R \times C^0(\mathbb R) \to \mathbb R & E(0, x, q_0) &= q_0(x)
\end{align}
Then, $E(t, x, q_0)$ needs to agree with the solution $q(t, x)$ to \eqref{eq:conslaw} up to terms $\mathcal O(t^3)$ (\cite{barsukow19activeflux}).
It is conceivable to solve, for instance, the locally linearized IVP 
\begin{align}
\del_t q + f'(q_1) \del_x q &= 0  & q(0, x) &= q_0(x)
\end{align}
by freezing the Jacobian $f'$ at some convenient value, e.g. $q_1 = q_0(x)$ as
\begin{align}
 E^{\text{loc.lin.}}(t, x, q_0) = q_0(x - f'(q_1)t) \label{eq:linearizedevoop}
\end{align}
but this would generally degrade the overall accuracy to second. It has been shown in \cite{barsukow19activeflux} that several iterations of this procedure, though, are a Picard iteration of the implicit characteristic equation for \eqref{eq:conslaw}, and that third-order accuracy is obtained from linearizing again around solution \eqref{eq:linearizedevoop}
\begin{align}
 E(t, x, q_0) = q_0\Big(x - f'\big(q_0(x - f'(q_1)t)\big)t\Big) \label{eq:updatepointsdetailedhighorder}
\end{align}

Traditional Active Flux (as opposed, for instance, to semi-discrete approaches \cite{abgrall20,abgrall22}) thus evolves the point values as
\begin{align}
 q_{i+\frac12}^{n+1} = E(\Delta t, x_{i+\frac12}, q^n_{\text{recon}}) \label{eq:updatepoints}
\end{align}
where $\Delta t = t^{n+1} - t^n$ not necessarily constant. Many of the details of the evolution operator will not be of great importance in what follows. It \emph{is} important, though, that it has some monotonicity properties. Below, it is assumed that it is a transport operator, i.e. that it takes a value of the reconstruction of the upwind cell and that $q_{i+\frac12}^{n+1}$ is contained between $ q_{i-\frac12}^{n}$ and $ q_{i+\frac12}^{n}$ if this cell is $C_i$ and the data are monotone (see Equation \eqref{eq:upwindevoop}). In this way, the evolution operator inherits monotonicity properties of the reconstruction.

The natural average update is by integrating \eqref{eq:conslaw} over the cell and over the time step
\begin{align}
\frac{\bar q_i^{n+1} - \bar q_i^n}{\Delta t} + \frac{1}{\Delta t} \int_{t^n}^{t^{n+1}}\frac{f(q(t, x_{i+\frac12})) - f(q(t, x_{i-\frac12}))}{\Delta x} \, \dd t &= 0 
\end{align}
To replace the integrals by quadratures one requires sufficiently many accurate quadrature points in time. To be exact for parabolas, an additional point value is required that is customarily placed at half the time step, and is obtained by using the same evolution operator on the same data:
\begin{align}
 q_{i+\frac12}^{n+\frac12} := E\left(\frac{\Delta t}{2}, x_{i+\frac12}; q^n_{\text{recon}}\right) \label{eq:updatepointshalf}
\end{align}
Then, the update of the average reads
\begin{align}
\bar q_i^{n+1} &= \bar q_i^n - \lambda \left( \hat f_{i+\frac12} - \hat f_{i-\frac12} \right ) \label{eq:updateaverage}\\
\hat f_{i+\frac12} &:= \alpha f(q_{i+\frac12}^{n}) + \beta f(q_{i+\frac12}^{n+\frac12}) + \gamma f(q_{i+\frac12}^{n+1}) \label{eq:quadratureflux}
\end{align}
where $\lambda := \frac{\Delta t}{\Delta x}$ and $\alpha, \beta, \gamma$ are the weights of Simpson's quadrature. In particular they are positive and fulfill $\alpha + \beta + \gamma=  1$. Equations \eqref{eq:updatepoints}, \eqref{eq:updatepointshalf}, \eqref{eq:updateaverage} and \eqref{eq:quadratureflux} constitute the Active Flux method. It is fully explicit and one-stage.

The presence of two degrees of freedom per cell is different from classical Finite Difference methods since the degrees of freedom are given different interpretations. One can employ the reconstruction to convert \eqref{eq:updateaverage} into an update equation for a point value at cell center, and thus obtain a method that evolves only point values -- but an interpretation of this method as a Finite Difference method would require handle different update formulas for every second point, which means that the point values can and must still be distinguished.

\begin{example} \label{ex:linadvincrement}
Consider linear advection
\begin{align}
 \del_t q + c \del_x q &= 0 \label{eq:linadv}
\end{align}
with $c>0$ and define the CFL number $\cfl = c \lambda$. One can rewrite for the case of a parabolic reconstruction
\begin{align}
 q_{i+\frac12}^{n+1} &= q_{i+\frac12}^n (1 - 4 \cfl + 3 \cfl^2) + q_{i-\frac12}^n (-2 \cfl + 3 \cfl^2) - 6 \cfl \bar q_i^n (\cfl - 1) \\
 &= q_{i+\frac12}^n - (q_{i+\frac12}^n - \bar q_i^n) \cfl (4 - 3 \cfl) - (\bar q_i^n - q_{i-\frac12}^n) \cfl (3 \cfl - 2)
\end{align}
and using Simpson's rule as flux quadrature
\begin{align}
 \bar q_i^{n+1} &= \bar q_i^n (1 + 2 \cfl)(1 - \cfl)^2 + \bar q_{i-1}^n \cfl^2 (3 - 2\cfl)- q_{i-\frac32}(1-\cfl)\cfl^2 + q_{i-\frac12}^n \cfl (1-\cfl) - q_{i+\frac12}^n (1-\cfl)^2 \cfl \\
 &= \bar q_i^n + \cfl^2 (1 - \cfl) (\bar q_{i-1}^n - q_{i-\frac32}^n) - \cfl( 1+ \cfl - \cfl^2) (\bar q_i^n - q_{i-\frac12}^n) \\&\qquad - \cfl^2 (2 - \cfl) (q_{i-\frac12}^n - \bar q_{i-1}^n) -\cfl(1 - \cfl)^2 (q_{i+\frac12}^n - \bar q_i^n)
\end{align}
\end{example}

\subsection{Limited reconstruction}

Define
\begin{align}
 R^n_i &:= \frac{q_{i+\frac12}^n - \bar q_i^n}{\bar q_i^n - q_{i-\frac12}^n} & &\Rightarrow & R^n_i+1 &:= \frac{q_{i+\frac12}^n - q_{i-\frac12}^n}{\bar q_i^n - q_{i-\frac12}^n}
\end{align}
$R^n_i > 0$ is equivalent to monotone data. Recall the following Theorem from \cite{barsukow19activeflux}:
\begin{theorem} \label{thm:parabola}
 The parabolic reconstruction \eqref{eq:parabolicrecon} is monotone iff $\frac{1}{2} \leq R_i^n \leq 2$. 
\end{theorem}

In the following a general reconstruction of the form 
\begin{align}
 q_{\text{recon},i}^n &= q_{i+\frac12}^n - (q_{i+\frac12}^n - q_{i-\frac12}^n) \rho^n_i(x) \label{eq:definitionrho}
\end{align}
will be considered, with $\rho_i^n$ depending on $x$ and $R_i^n$ only and satisfying
$\rho^n_i\left(-\frac{\Delta x}{2}\right) = 1$, $\rho^n_i\left(\frac{\Delta x}{2}\right) = 0$, and
\begin{align}
 \int_{-\frac{\Delta x}{2}}^{\frac{\Delta x}{2}}\rho^n_i(x) \, \dd x &= \Delta x \frac{q_{i+\frac12}^n - \bar q_i^n}{q_{i+\frac12}^n - q_{i-\frac12}^n} = \Delta x \frac{R_i^n}{R_i^n + 1}
\end{align}

For the parabola one finds
\begin{align}
 \rho_i^n(x) &= \left(\frac12-\frac{x}{\Delta x}\right)\frac{6 \frac{x}{\Delta x}(R_i^n-1) + 5 R_i^n - 1}{2(R_i^n+1)}\label{eq:parabolarho}
\end{align}
and for the power law reconstruction as introduced in \cite{barsukow19activeflux} 
\begin{align}
\rho^n_i(x) &= \begin{cases} 1 - \left(\frac{1}{2}  + \frac{x}{\Delta x}\right)^{R^n_i} & R_i^n > 2\\
                            \left(\frac{1}{2} - \frac{x}{\Delta x}\right)^{1/R^n_i} &  0 < R_i^n < \frac12        \end{cases} \label{eq:powerlaw}
               \end{align}
Examples of the power laws \eqref{eq:powerlaw} are shown in Figure \ref{fig:rho}. The combination of \eqref{eq:powerlaw} (for the intervals of $R_i^n$ given), and otherwise \eqref{eq:parabolarho} (if $\frac12 \leq R_i^n \leq 2$) allows to have what was called earlier a monotonized reconstruction, i.e. a monotone $\rho_i^n$ with $0 \leq \rho_i^n(x) \leq 1$ for all $x \in \left[ - \frac{\Delta x}{2}, \frac{\Delta x}{2} \right]$ and all $R_i^n \in (0, \infty)$.

\begin{figure}
 \centering
 \includegraphics[width=0.5\textwidth]{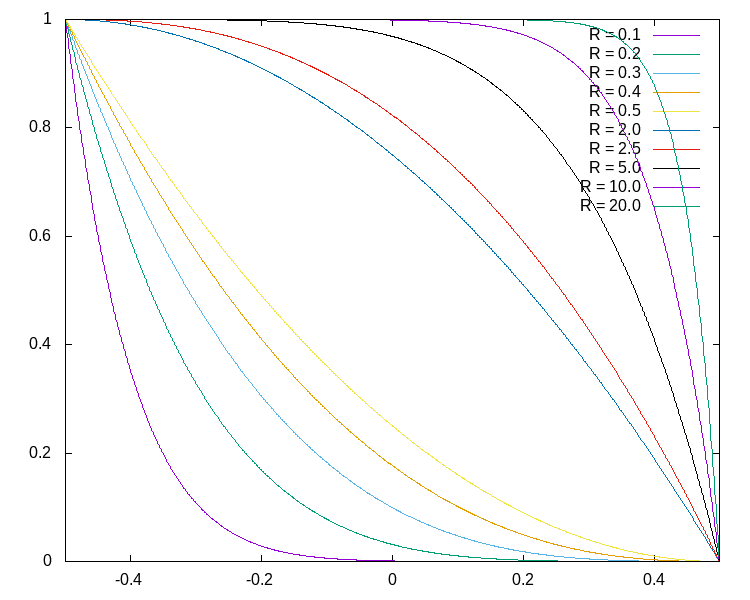}
 \caption{The functions $\rho^n_i$ as defined in \eqref{eq:powerlaw} for various values of $R_i^n$, with $\Delta x = 1$.}
 \label{fig:rho}
\end{figure}

\section{Generalization of Harten's lemma} \label{sec:harten}

Inspired by Example \ref{ex:linadvincrement} consider the update formulas of Active Flux to be in increment form such that
\begin{align}
 q_{i+\frac12}^{n+1} - \bar q_i^{n+1} &= (1 + A_i) (q_{i+\frac12}^n - \bar q_i^n) + B_i (\bar q_i^n - q_{i-\frac12}^n) + C_i (q_{i-\frac12}^n - \bar q_{i-1}^n) + D_i (\bar q_{i-1}^n - q_{i-\frac32}^n)  \label{eq:updatedifference1}\\
 \bar q_i^{n+1} - q_{i-\frac12}^{n+1} &= \tilde A_i (q_{i+\frac12}^n - \bar q_i^n) + (1 + \tilde B_i) (\bar q_i^n - q_{i-\frac12}^n) + \tilde C_i (q_{i-\frac12}^n - \bar q_{i-1}^n) + \tilde D_i (\bar q_{i-1}^n - q_{i-\frac32}^n)  \label{eq:updatedifference2}
\end{align}
for some $A_i, B_i, C_i, D_i, \tilde A_i, \tilde B_i, \tilde C_i, \tilde D_i$ that, of course, are allowed to depend on $q^n$.

\begin{example} \label{ex:linadvincrement2}
Continuing Example \ref{ex:linadvincrement}, one finds for standard Active Flux with parabolic reconstruction:
\begin{align}
 A_i &= \cfl (3 \cfl - 4) +\cfl(1 - \cfl)^2 =  \cfl^2 - 3\cfl   + \cfl^3\\
 B_i &= \cfl (2 - 3 \cfl) + \cfl( 1+ \cfl - \cfl^2) =  3\cfl - 2 \cfl^2 - \cfl^3 
\end{align}
\begin{align}
 D_i &= - \cfl^2 (1 - \cfl) &
 C_i &= \cfl^2 (2 - \cfl)  \\
 \tilde A_i &= -\cfl(1 - \cfl)^2 &
 \tilde B_i &= - \cfl( 1+ \cfl - \cfl^2) 
\end{align}
\begin{align}
 \tilde C_i &= \cfl (4 - 3 \cfl) - \cfl^2 (2 - \cfl) = 4\cfl - 5 \cfl^2  + \cfl^3 \\
 \tilde D_i &= \cfl (3 \cfl - 2) + \cfl^2 (1 - \cfl) = 4 \cfl^2 - 2\cfl - \cfl^3
\end{align}
\end{example}

The most natural generalization of Harten's lemma is the following.

\begin{lemma}
 Assume update formulas \eqref{eq:updatedifference1}--\eqref{eq:updatedifference2}, as well as all $1+A_i, B_i, C_i, D_i, \tilde A_i ,1 + \tilde B_i,\tilde C_i , \tilde D_i  \geq 0$, and
 \begin{align}
  A_i + \tilde A_i + C_{i+1} + \tilde C_{i+1} &\leq 0  &
  B_i + \tilde B_i + D_{i+1} + \tilde D_{i+1} &\leq 0
 \end{align}
 
 Then $\mathrm{TV}(q^{n+1}) \leq \mathrm{TV}(q^n)$.
\end{lemma}
The proof (via the triangle inequality) is omitted because this Lemma is not useful: its conditions are never true for the parabolic reconstruction in conjunction with Simpson's rule (see Example \ref{ex:linadvincrement2}), e.g. $D_i$ is not positive. Thus the Lemma does not let us conclude anything, and better sufficient conditions are needed.

Observe that in an inequality such as $x + y + z \geq 0$ one does not need to choose, as sufficient condition, 
$ x \geq 0$, $y\geq 0$, $z \geq 0$ individually, but can for example take $x + y \geq 0$, $z \geq 0$. 
Combine therefore the terms in \eqref{eq:updatedifference1}--\eqref{eq:updatedifference2} in the following way:
\begin{align}
 q_{i+\frac12}^{n+1} - \bar q_i^{n+1} &=  \left( (1 + A_i) R_i^n + B_i \right) (\bar q_i^n - q_{i-\frac12}^n) + (C_i R_{i-1}^n + D_i) (\bar q_{i-1}^n - q_{i-\frac32}^n)  \\
 \bar q_i^{n+1} - q_{i-\frac12}^{n+1} &= ( \tilde A_i R_i^n + 1 + \tilde B_i) (\bar q_i^n - q_{i-\frac12}^n) + (\tilde C_i R_{i-1}^n  + \tilde D_i) (\bar q_{i-1}^n - q_{i-\frac32}^n)  
\end{align}
since $R_i^n (\bar q_i^n - q_{i-\frac12}^n) = q_{i+\frac12}^n - \bar q_i^n$.

\begin{lemma} \label{lem:harten1}
 Assume update formulas \eqref{eq:updatedifference1}--\eqref{eq:updatedifference2} with 
 \begin{align}
  (1 + A_i) R_i^n + B_i &\geq 0 &
  C_i R_{i-1}^n + D_i &\geq 0 \\
  \tilde A_i R_i^n + 1 + \tilde B_i &\geq 0 & \tilde C_i R_{i-1}^n  + \tilde D_i &\geq 0
 \end{align}
as well as
 \begin{align}
  -1 \leq A_i+ C_{i+1} + \tilde A_i  +  \tilde C_{i+1} &\leq 0  &
  B_i + D_{i+1}  + \tilde B_i   + \tilde D_{i+1} &\leq 0
 \end{align}

 Then $\mathrm{TV}(q^{n+1}) \leq \mathrm{TV}(q^n)$.
\end{lemma}
\begin{proof}
 By triangle inequality
 \begin{align}
  \mathrm TV(q^{n+1}) &= \sum_i |q_{i+\frac12}^{n+1}- \bar q_i^{n+1} | + \sum_i | \bar q_i^{n+1} - q_{i-\frac12}^{n+1}| \\
   &\leq\sum_i \Big( \left| (1 + A_i) R_i^n + B_i \right| |\bar q_i^n - q_{i-\frac12}^n| + |C_i R_{i-1}^n + D_i| |\bar q_{i-1}^n - q_{i-\frac32}^n| \nonumber \\
  &\qquad \qquad\qquad \qquad  + | \tilde A_i R_i^n + 1 + \tilde B_i| |\bar q_i^n - q_{i-\frac12}^n| + |\tilde C_i R_{i-1}^n  + \tilde D_i| |\bar q_{i-1}^n - q_{i-\frac32}^n| \Big )\\
  &=\sum_i \Big( \left( (1 + A_i+ C_{i+1} + \tilde A_i  +  \tilde C_{i+1}  )R_i^n \right) |\bar q_{i}^n - q_{i-\frac12}^n| \nonumber \\
  & \qquad \qquad\qquad \qquad + ( 1+ B_i + D_{i+1}  + \tilde B_i   + \tilde D_{i+1}) |\bar q_{i}^n - q_{i-\frac12}^n|  \Big )\\
  &\leq \sum_i \Big(  (1 + A_i+ C_{i+1} + \tilde A_i  +  \tilde C_{i+1}  ) |\bar q_{i+\frac12}^n - \bar q_{i}^n| \nonumber \\
  &\qquad \qquad\qquad \qquad  + ( 1+ B_i + D_{i+1}  + \tilde B_i   + \tilde D_{i+1}) |\bar q_{i}^n - q_{i-\frac12}^n|  \Big )\\
  &\leq \mathrm{TV}(q^n)
 \end{align}

\end{proof}

The following result demonstrates that Lemma \ref{lem:harten1} is indeed useful, since it allows to prove the TVD property for Active Flux in the case of a monotone parabolic reconstruction:

\begin{theorem}
 Assume $\frac12 \leq  R_i \leq 2$. Then Active Flux for linear advection \eqref{eq:linadv} with a parabolic reconstruction, Simpson's rule for the flux quadrature and subject to the usual CFL condition is TVD.
\end{theorem}
\begin{proof}
By assumption and Theorem \ref{thm:parabola} the parabolic reconstruction is monotone. Recall that $\cfl = c \lambda > 0$, $\cfl \leq 1$.
The coefficients in the update formulas \eqref{eq:updatedifference1}--\eqref{eq:updatedifference2}  are given in Example \ref{ex:linadvincrement2}.
 According to Lemma \ref{lem:harten1} one needs to show the non-negativity of the coefficients:
 \begin{align}
 (1 + A_i) R_i^n + B_i &= (1 + \cfl^2 - 3\cfl + \cfl^3) R_i^n + 3\cfl - 2 \cfl^2 - \cfl^3 \geq 0\\
 C_i R_{i-1}^n + D_i &= \cfl^2 \Big ( (2 - \cfl) R_{i-1}^n - (1 - \cfl) \Big ) \geq 0\\
 \tilde A_i R_i^n + 1 - \tilde B_i &= -\cfl(1 - \cfl)^2 R_i^n + 1 - \cfl( 1+ \cfl - \cfl^2)  \geq 0\\
 \tilde C_i R_{i-1}^n  + \tilde D_i &= (4\cfl - 5 \cfl^2  + \cfl^3)R_{i-1}^n + 4 \cfl^2 - 2\cfl - \cfl^3 \geq 0
 \end{align}
 
 To this end one should first observe that
 \begin{align}
  1 + \cfl^2 - 3\cfl + \cfl^3 &\begin{cases} > 0 & \cfl < \sqrt{2}-1 \\
                               < 0 & \cfl > \sqrt{2}-1
                              \end{cases} \\
 4\cfl - 5 \cfl^2  + \cfl^3 &\geq 0  \qquad \cfl \in [0,1]
 \end{align}
 
 One then finds
 \begin{align}
 R_i^n &\begin{cases}\displaystyle \geq \frac{- 3\cfl + 2 \cfl^2 + \cfl^3}{1 + \cfl^2 - 3\cfl + \cfl^3} \in (-\infty, 0] & 0 \leq \cfl < \sqrt{2}-1 \\ \displaystyle \leq \frac{- 3\cfl + 2 \cfl^2 + \cfl^3}{1 + \cfl^2 - 3\cfl + \cfl^3} \in [2, \infty) & \sqrt{2}-1 < \cfl \leq 1 \end{cases} \\
 R_{i-1}^n &\geq \frac{1 - \cfl}{2 - \cfl} \in \left[0 , \frac12\right]\\
  R_i^n &\leq \frac{ 1 - \cfl( 1+ \cfl - \cfl^2) }{\cfl(1 - \cfl)^2} \in [2, \infty] \\
 R_{i-1}^n &\geq \frac{- 4 \cfl + 2 + \cfl^2 }{4 - 5 \cfl  + \cfl^2} \in \left(-\infty, \frac12\right]
 \end{align}
 
 These inequalities are true, because by assumption $\frac12 \leq R_i^n \leq 2$.
 
\end{proof}

Instead of \eqref{eq:updatedifference1}--\eqref{eq:updatedifference2} it is helpful to consider also the following update formulas
\begin{align}
 q_{i+\frac12}^{n+1} - \bar q_i^{n+1} &= q_{i+\frac12}^n - \bar q_i^n + S_{i+1} (q_{i+\frac32}^n - q_{i+\frac12}^n) + P_i (q_{i+\frac12}^n - q_{i-\frac12}^n) + Q_{i-1} (q_{i-\frac12}^n - q_{i-\frac32}^n) \label{eq:updatedifference1new}\\
 \bar q_i^{n+1} - q_{i-\frac12}^{n+1} &= \bar q_i^n - q_{i-\frac12}^n   + \tilde S_{i+1} (q_{i+\frac32}^n - q_{i+\frac12}^n) + \tilde P_i (q_{i+\frac12}^n - q_{i-\frac12}^n) + \tilde Q_{i-1} (q_{i-\frac12}^n - q_{i-\frac32}^n)\label{eq:updatedifference2new}
\end{align}
and the corresponding generalization of Harten's lemma:
\begin{lemma} \label{lem:harten}
  Assume update formulas \eqref{eq:updatedifference1new}--\eqref{eq:updatedifference2new} with 
 \begin{align}
   R_i^n + P_i( R_i^n  + 1)&\geq 0 &
   S_{i+1} (R_{i+1}^n + 1)&\geq 0 &
   Q_{i-1} (R_{i-1}^n + 1)&\geq 0 \\  
  1 + \tilde P_i (R_i^n + 1) &\geq 0 &
  \tilde S_{i+1} (R_{i+1}^n + 1) & \geq 0 &
  \tilde Q_{i-1}(R_{i-1}^n + 1) & \geq 0
 \end{align}
as well as
 \begin{align}
  -1 \leq  S_{i}+ P_i  + Q_{i} + \tilde S_{i} + \tilde P_i     + \tilde Q_{i}&\leq 0 \label{eq:consistencycondlemma}
 \end{align}

 Then $\mathrm{TV}(q^{n+1}) \leq \mathrm{TV}(q^n)$.
\end{lemma}
\begin{proof}
 Rewrite \eqref{eq:updatedifference1new}--\eqref{eq:updatedifference2new} first as
\begin{align*}
 q_{i+\frac12}^{n+1} - \bar q_i^{n+1} &= (1 + P_i) (q_{i+\frac12}^n - \bar q_i^n) + P_i (\bar q_i^n - q_{i-\frac12}^n) + S_{i+1} (q_{i+\frac32}^n - q_{i+\frac12}^n)  + Q_{i-1} (q_{i-\frac12}^n - q_{i-\frac32}^n) \\
 &= \big((1 + P_i) R_i^n  + P_i\big) (\bar q_i^n - q_{i-\frac12}^n) + S_{i+1} (R_{i+1}^n + 1) (\bar q_{i+1}^n - q_{i+\frac12}^n) \\&\qquad \qquad \nonumber   + Q_{i-1} (R_{i-1}^n + 1) (\bar q_{i-1}^n - q_{i-\frac32}^n) \\
 \bar q_i^{n+1} - q_{i-\frac12}^{n+1} &= \big(1 + \tilde P_i (R_i^n + 1)\big) (\bar q_i^n - q_{i-\frac12}^n)  + \tilde S_{i+1} (R_{i+1}^n + 1) (\bar q_{i+1}^n - q_{i+\frac12}^n) \\&\qquad \qquad \nonumber + \tilde Q_{i-1}(R_{i-1}^n + 1) (\bar q_{i-1}^n - q_{i-\frac32}^n)
\end{align*}
Then, by triangle inequality
 \begin{align}
  \mathrm TV(q^{n+1}) 
   &\leq  \sum_i \Big((1 + P_i) R_i^n  + P_i\Big) |\bar q_i^n - q_{i-\frac12}^n| + S_{i+1} (R_{i+1}^n + 1) |\bar q_{i+1}^n - q_{i+\frac12}^n|  \\&\qquad\qquad \nonumber + Q_{i-1} (R_{i-1}^n + 1) |\bar q_{i-1}^n - q_{i-\frac32}^n| +  \Big(1 + \tilde P_i (R_i^n + 1)\Big ) |\bar q_i^n - q_{i-\frac12}^n|  \\&\qquad\qquad \nonumber + \tilde S_{i+1} (R_{i+1}^n + 1) |\bar q_{i+1}^n - q_{i+\frac12}^n| + \tilde Q_{i-1}(R_{i-1}^n + 1) |\bar q_{i-1}^n - q_{i-\frac32}^n|\\
   & = \sum_i (1  + S_{i}+ P_i  + Q_{i} + \tilde S_{i} + \tilde P_i     + \tilde Q_{i}) |q_{i+\frac12}^n-\bar q_i^n| \\& \qquad \qquad \nonumber + \sum_i (1 + S_{i}+ P_i+ Q_{i}  + \tilde S_{i}  + \tilde P_i   + \tilde Q_{i}       ) |\bar q_i^n - q_{i-\frac12}^n|  \\
  &\leq \mathrm{TV}(q^n)
 \end{align}

\end{proof}

\section{Special cases of situations requiring flux limiting} \label{sec:examples}

In this Section, consider linear advection \eqref{eq:linadv} with $c>0$ for simplicity.

\subsection{Right-moving jump} \label{ssec:examplesrightmovjump}

Consider monotone data and in some cell $i$ the average to be
\begin{align}
 \bar q_i^n = q_{i-\frac12}^n + \epsilon \qquad \epsilon > 0
\end{align}
Assume moreover that $q_{i-\frac12}^{n+1} = q_{i-\frac12}^{n+\frac12} = q_{i-\frac12}^n$ which is the case if the data are constant to the left of $x_{i-\frac12}$, see Figure \ref{fig:example1}, top left.

\begin{figure}
 \centering
 \includegraphics[width=\textwidth]{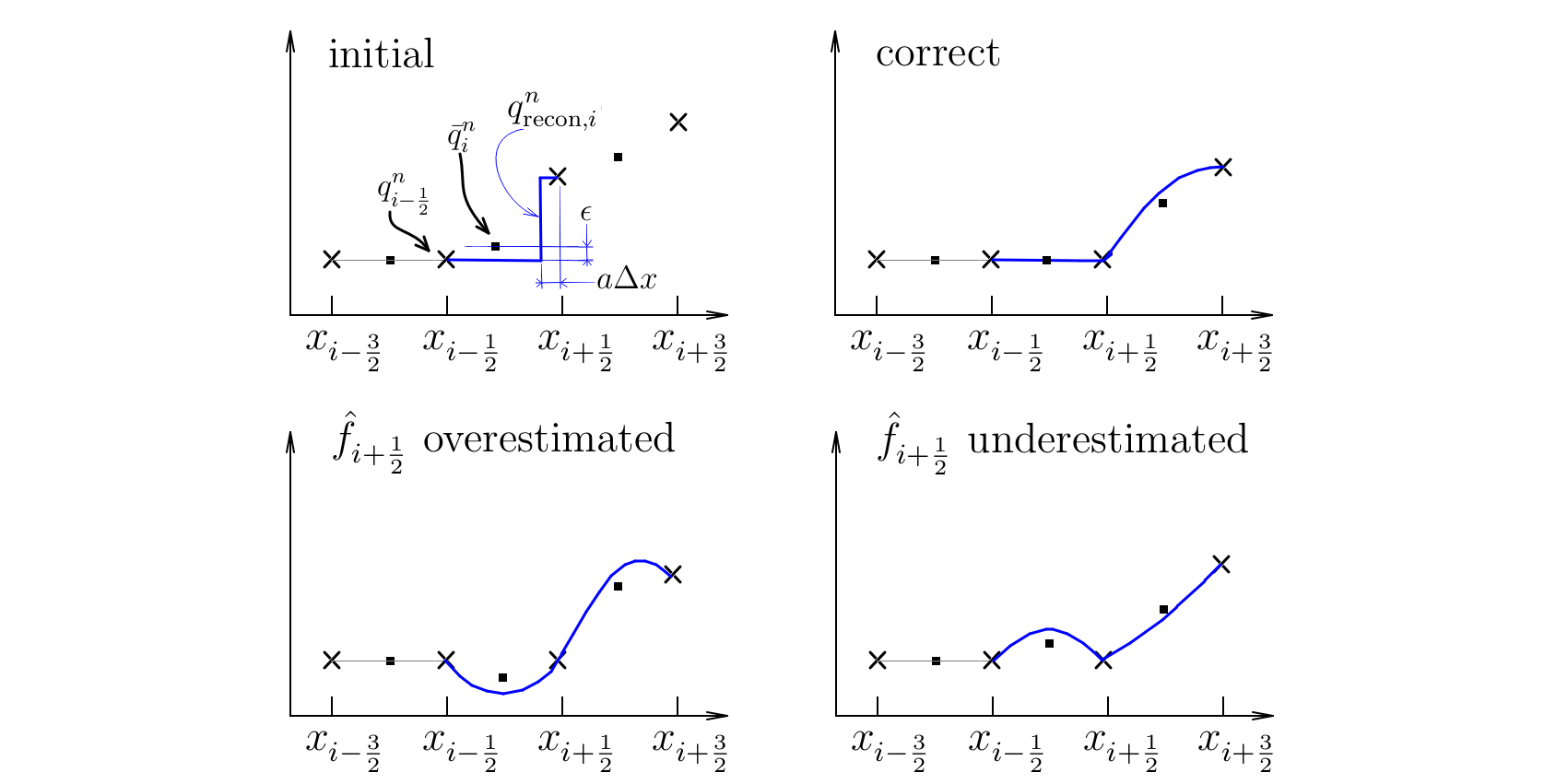}
 \caption{\emph{Top left}: Initial data and a reconstruction in cell $C_i$. \emph{Top right}: Exact evolution over one time step remains monotone. \emph{Bottom}: The assumed choice of a large CFL implies $q_{i+\frac12}^{n+1} = q_{i-\frac12}^n$, and then there remains only one value of the flux that leads to a monotone solution. The results of over/underestimating it are shown.}
 \label{fig:example1}
\end{figure}

Since the data are monotone, there exists a continuous monotone reconstruction that is (e.g. in the $L^2$ sense) arbitrarily close to the following (discontinuous) one:
\begin{align}
 q_{\text{recon},i}^n(x) = \begin{cases} q_{i-\frac12}^n & x < \frac{\Delta x}{2} - a \Delta x \\ q_{i+\frac12}^n & x > \frac{\Delta x}{2} - a \Delta x  \end{cases}   \qquad x \in \left[-\frac{\Delta x}{2}, \frac{\Delta x}{2}\right]
\end{align}
which is conservative if $a = \frac{\epsilon}{q_{i+\frac12}^n-q_{i-\frac12}^n}$.
For small $\epsilon$ and a CFL number $\cfl := c \lambda$ greater than $2a$ one will find $q_{i+\frac12}^{n+1} = q_{i+\frac12}^{n+\frac12} = q_{i-\frac12}^n$. 
Assume now that the flux is computed using the standard quadrature
\begin{align}
 \hat f_{i+\frac12} = c \left( \alpha q_{i+\frac12}^{n} + \beta q_{i+\frac12}^{n+\frac12} + \gamma q_{i+\frac12}^{n+1} \right)
\end{align}
with $\alpha + \beta + \gamma = 1$.
Then, in the situation described above,
\begin{align}
 \bar q_i^{n+1} &= \bar q_i^n - \cfl \left( \alpha q_{i+\frac12}^n + (\beta + \gamma) q_{i-\frac12}^n  - q_{i-\frac12}^n  \right) = \bar q_i^n - \cfl \alpha \left(  q_{i+\frac12}^n - q_{i-\frac12}^n  \right)
\end{align}

 Then, enforcing monotonicity \eqref{eq:monoton} at the new time step,

\begin{align}
 q_{i+\frac12}^{n+1} - \bar q_i^{n+1} &= q_{i-\frac12}^{n} - \bar q_i^n + \cfl \alpha \left(  q_{i+\frac12}^n - q_{i-\frac12}^n  \right)  \geq 0\\
 \cfl \alpha &\geq \frac{\bar q_i^n - q_{i-\frac12}^{n}}{q_{i+\frac12}^n - q_{i-\frac12}^n} = \frac{\epsilon}{q_{i+\frac12}^n - q_{i-\frac12}^n}
\end{align}

as well as

\begin{align}
 \bar q_i^{n+1} - q_{i-\frac12}^{n+1} = \bar q_i^n - q_{i-\frac12}^{n}  - \cfl \alpha \left(  q_{i+\frac12}^n - q_{i-\frac12}^n  \right)  &\geq 0\\
 \frac{ \epsilon }{q_{i+\frac12}^n - q_{i-\frac12}^n} = \frac{ \bar q_i^n - q_{i-\frac12}^{n} }{q_{i+\frac12}^n - q_{i-\frac12}^n} &\geq \cfl \alpha 
\end{align}

One thus faces the situation that there is a \emph{unique} choice of $\alpha$ (and thus of the flux $\hat f_{i+\frac12}$), namely
\begin{align}
 \cfl \alpha = \frac{\epsilon}{q_{i+\frac12}^n - q_{i-\frac12}^n}
\end{align}
Having assumed $\cfl$ big (close to 1), $\alpha$ therefore needs to be chosen small. The full flux is
\begin{align}
 \hat f_{i+\frac12}/c &= \alpha q_{i+\frac12}^n + (\beta + \gamma) q_{i-\frac12}^n \\
 &= \frac{\epsilon}{\cfl(q_{i+\frac12}^n - q_{i-\frac12}^n)} q_{i+\frac12}^n + \left(1 - \frac{\epsilon}{\cfl(q_{i+\frac12}^n - q_{i-\frac12}^n)} \right) q_{i-\frac12}^n \\
 &= q_{i-\frac12}^n + \frac{\epsilon}{\cfl}
\end{align}

The uniqueness of the choice can be explained as follows (see also Figure \ref{fig:example1}, bottom): The flux on the left cell interface (at $x_{i-\frac12}$) is exact. The exact flux on the right cell interface is high for a time period $$\tau = \frac{a \Delta x}{c} = \frac{\epsilon \Delta x}{c( q_{i+\frac12}^n-q_{i-\frac12}^n)} = \Delta t \frac{\epsilon}{\cfl ( q_{i+\frac12}^n-q_{i-\frac12}^n) }$$
and then equals to the left flux for the rest of the time step. If the flux $\hat f_{i+\frac12}$ estimated via the quadrature is a little larger than the time-average of the exact flux, there will be an undershoot $\bar q_i^{n+1} < q_{i-\frac12}$. One thus does not want to put a large weight on the large value $q_{i+\frac12}^n$, such that naturally $\alpha$ should be small, and needs to vanish if $\epsilon \to 0$. If the flux $\hat f_{i+\frac12}$ is a little less than the time-average of the exact flux, then the new average $\bar q_i^{n+1}$ will not drop to $q_{i-\frac12}^{n+1} = q_{i-\frac12}^{n}$, while the right point value will do so. Therefore the fact that the flux $\hat f_{i+\frac12}$ has a unique value has to do with the fact that in the situation described above there is just one TVD value of the average. It is noteworthy that $\hat f_{i+\frac12}$ depends on the CFL number $\cfl$. Below, a modified flux will be computed not by modifying the quadrature weights, but by changing the half-time-step value $f(q_{i+\frac12}^{n+\frac12})$, which at least for linear advection is perfectly equivalent.

\subsection{CFL number is 1} \label{ssec:examplesnuone}

Consider now the case of a CFL number $\cfl = 1$, and for simplicity use $\beta = 0$, i.e. a quadrature similar to the trapezoidal rule. Then the reconstruction does not play a role, and
\begin{align}
 q_{i+\frac12}^{n+1} &= q_{i-\frac12}^{n} & q_{i-\frac12}^{n+1} &= q_{i-\frac32}^{n} 
\end{align}
such that
\begin{align}
 \bar q_i^{n+1} = \bar q_i^n - \left( \alpha_{i+\frac12} q_{i+\frac12}^n - \alpha_{i-\frac12} q_{i-\frac12}^n + \gamma_{i+\frac12} q_{i-\frac12}^n - \gamma_{i-\frac12} q_{i-\frac32}^n  \right)
\end{align}
Observe that here, for clarity of presentation, the quadrature weights are assumed to be different from cell interface to cell interface. Assume monotone data $q_{i-\frac12}^n \leq \bar q_i^n \leq q_{i+\frac12}^n$; monotonicity at the new time means
\begin{align}
 \bar q_i^{n+1} - q_{i-\frac12}^{n+1} = \bar q_i^n  - q_{i-\frac32}^{n} - \left( \alpha_{i+\frac12} q_{i+\frac12}^n - \alpha_{i-\frac12} q_{i-\frac12}^n + \gamma_{i+\frac12} q_{i-\frac12}^n - \gamma_{i-\frac12} q_{i-\frac32}^n  \right) \geq 0\\
 q_{i+\frac12}^{n+1} - \bar q_i^{n+1} = q_{i-\frac12}^{n} - \bar q_i^n  + \left( \alpha_{i+\frac12} q_{i+\frac12}^n - \alpha_{i-\frac12} q_{i-\frac12}^n + \gamma_{i+\frac12} q_{i-\frac12}^n - \gamma_{i-\frac12} q_{i-\frac32}^n  \right) \geq 0
\end{align}
Consider the special case $q_{i-\frac12}^n = q_{i-\frac32}^n = 0$:
\begin{align}
 \bar q_i^n  -  \alpha_{i+\frac12} q_{i+\frac12}^n   \geq 0\\
 - \bar q_i^n  + \alpha_{i+\frac12} q_{i+\frac12}^n  \geq 0
\end{align}
In this case, again, there is only one possible value for $\alpha_{i+\frac12}$:
\begin{align}
 \alpha_{i+\frac12} = \frac{\bar q_i^n}{q_{i+\frac12}^n }
\end{align}
since, whatever the point value, the entire mass $\Delta x \bar q_i^n$ flows out of the cell over the time step. Below, for practical reasons, a modified (limited) flux $\hat f_{i+\frac12}$ is obtained by choosing a different $f(q_{i+\frac12}^{n+\frac12})$ and not by modifying quadrature weights.

\section{A TVD Active Flux method} \label{sec:TVNI}

Consider the scalar conservation law \eqref{eq:conslaw}. Assume that the point value update \eqref{eq:updatepoints} amounts to an evaluation of the reconstruction at some location 
such that
\begin{align}
 q_{i+\frac12}^{n+1} &=q^n_{\text{recon}}\left(x_{i+\frac12} - \sigma^{n+1}_{i+\frac12} \Delta t \right ) = \begin{cases} q^n_{\text{recon},i}\left( \frac{\Delta x}{2} - \sigma^{n+1}_{i+\frac12}  \Delta t \right ) & \sigma_{i+\frac12}^{n+1} > 0 \\
                        q^n_{\text{recon},i+1}\left( -\frac{\Delta x}{2} - \sigma^{n+1}_{i+\frac12} \Delta t \right ) & \sigma_{i+\frac12}^{n+1} \leq 0
                       \end{cases} \label{eq:upwindevoop}\\
        &\overset{ \eqref{eq:definitionrho}}{=} \begin{cases} q_{i+\frac12}^n - (q_{i+\frac12}^n - q_{i-\frac12}^n) \rho_i^n\left( \frac{\Delta x}{2} - \sigma^{n+1}_{i+\frac12}  \Delta t \right ) & \sigma_{i+\frac12}^{n+1} > 0 \\
                        q_{i+\frac32}^n - (q_{i+\frac32}^n - q_{i+\frac12}^n) \rho_{i+1}^n\left( -\frac{\Delta x}{2} - \sigma^{n+1}_{i+\frac12}  \Delta t \right ) & \sigma_{i+\frac12}^{n+1} \leq 0
                       \end{cases}
\end{align}
where $\sigma^{n+1}_{i+\frac12}$ is an estimate of $f'(q)$ at $x_{i+\frac12}$ at $t^n + \Delta t$, e.g. as in \eqref{eq:updatepointsdetailedhighorder}. The last expression can be rewritten as
\begin{align}
 q_{i+\frac12}^n + (q_{i+\frac32}^n-q_{i+\frac12}^n)\left(1 -  \rho_{i+1}^n\left( -\frac{\Delta x}{2} - \sigma_{i+\frac12}  \Delta t \right )\right )
\end{align}

Write therefore generically
\begin{align}
 q_{i+\frac12}^{n+1} =: q_{i+\frac12}^n - \Delta q_{i+\frac12}^{n+1,\text{upw}} \rho_{i+\frac12}^{n+1} \label{eq:defofrho}
\end{align}
with 
\begin{align}
 \Delta q_{i+\frac12}^{n+1,\text{upw}} &:= \begin{cases} q_{i+\frac12}^n - q_{i-\frac12}^n& \sigma^{n+1}_{i+\frac12} > 0 \\ -(q_{i+\frac32}^n - q_{i+\frac12}^n) & \sigma^{n+1}_{i+\frac12} \leq 0 \end{cases} \label{eq:deltaqdef}\\
 &= (q_{i+\frac12}^n - q_{i-\frac12}^n) (\Sigma^{n+1}_{i+\frac12})^+ + (q_{i+\frac32}^n - q_{i+\frac12}^n) (\Sigma^{n+1}_{i+\frac12})^-
\end{align}
having defined $ (\Sigma^{n+s}_{i+\frac12})^\pm = (\sgn \, \sigma^{n+s}_{i+\frac12})^\pm \in \{ \pm1, 0 \}$
and some $\rho_{i+\frac12}^{n+1} \in [0,1]$ if the data are monotone and the reconstruction monotonized.

There is no analogous definition for $q_{i+\frac12}^{n+\frac12}$: define instead an auxiliary intermediate flux
\begin{align}
 f^{n+\frac12}_{i+\frac12} := f(q_{i+\frac12}^n) + b_{i+\frac12} (q_{i+\frac12}^{n+1} - q_{i+\frac12}^n) \label{eq:auxflux}
\end{align}
where $b_{i+\frac12}$ is a parameter to be determined in the limiting procedure.
The first term merely ensures that this is a consistent flux (if the data is constant it becomes the physical flux). The second term is chosen with the idea that $q_{i+\frac12}^n = q_{i+\frac12}^{n+1}$, at least for monotone data, implies that they are constant, which then is consistent with $f^{n+\frac12}_{i+\frac12} = f(q_{i+\frac12}^n)$. It thus makes sure that $b_{i+\frac12}$ is well-defined. One can, in principle, choose other ways of parametrizing the auxiliary flux $f_{i+\frac12}^{n+\frac12}$, but \eqref{eq:auxflux} seems to have some nice properties. Observe that using \eqref{eq:defofrho} the auxiliary flux is rewritten as
\begin{align}
 f^{n+\frac12}_{i+\frac12} := f(q_{i+\frac12}^n) - b_{i+\frac12}  \Delta q_{i+\frac12}^{n+1,\text{upw}} \rho_{i+\frac12}^{n+1} \label{eq:auxflux2}
\end{align}

The numerical flux $\hat f_{i+\frac12}$ is obtained via Simpson's rule
\begin{align}
 \hat f_{i+\frac12} &= \alpha f(q_{i+\frac12}^n) + \beta f^{n+\frac12}_{i+\frac12} + \gamma f(q_{i+\frac12}^{n+1}) \label{eq:intercellfluxquadrature}\\
 &= (\alpha  + \beta) f(q_{i+\frac12}^n) - \beta b_{i+\frac12}  \Delta q_{i+\frac12}^{n+1,\text{upw}} \rho_{i+\frac12}^{n+1}
 + \gamma f\Big(q_{i+\frac12}^n - \Delta q_{i+\frac12}^{n+1,\text{upw}} \rho_{i+\frac12}^{n+1} \Big)
\end{align}

In order to convert increments in flux into increments in $q$
define $a_i^n$ and $a_i^{n+1}$ as
\begin{align}
 f(q_{i+\frac12}^{n+s}) - f(q_{i-\frac12}^{n+s}) =: a_i^{n+s} (q_{i+\frac12}^{n+s}  - q_{i-\frac12}^{n+s} ) \qquad s = 0, 1
\end{align}
In case $q_{i+\frac12}^{n+s}  - q_{i-\frac12}^{n+s}$ vanishes, choose e.g. $a_i^{n+s} = f'(q_{i+\frac12}^{n+s})$. Observe that the definition \eqref{eq:auxflux} of the auxiliary flux is subtly different.
With these notions one has
\begin{align}
 \hat f_{i+\frac12} - \hat f_{i-\frac12} &= 
 ((\alpha + \beta)a_i^n  +  \gamma a_i^{n+1})              (q_{i+\frac12}^n - q_{i-\frac12}^n)
 \nonumber\\ 
  &+ \beta \Big(
 -b_{i+\frac12}  \Delta q_{i+\frac12}^{n+1,\text{upw}} \rho_{i+\frac12}^{n+1} + b_{i-\frac12}  \Delta q_{i-\frac12}^{n+1,\text{upw}} \rho_{i-\frac12}^{n+1}
 \Big) \nonumber\\ 
  &+ \gamma a_i^{n+1} \Big( - \Delta q_{i+\frac12}^{n+1,\text{upw}} \rho_{i+\frac12}^{n+1}  + \Delta q_{i-\frac12}^{n+1,\text{upw}} \rho_{i-\frac12}^{n+1} \Big)
\end{align}

Thus, using $\lambda := \frac{\Delta t}{\Delta x}$ and calling $m_i := (\alpha + \beta)a_i^n  +  \gamma a_i^{n+1}$,
{\footnotesize
 \begin{align}
  &q_{i+\frac12}^{n+1} - \bar q_i^{n+1} = q_{i+\frac12}^n -\bar q_i^n + \lambda m_i (q_{i+\frac12}^n - q_{i-\frac12}^n) 
  \\ \nonumber &   \qquad 
  -\big(1 + \lambda (\beta b_{i+\frac12} + \gamma a_i^{n+1})\big) \Delta q_{i+\frac12}^{n+1,\text{upw}} \rho_{i+\frac12}^{n+1}
  + \lambda (\beta b_{i-\frac12} + \gamma a_i^{n+1}) \Delta q_{i-\frac12}^{n+1,\text{upw}} \rho_{i-\frac12}^{n+1}  \\
  &= q_{i+\frac12}^n -\bar q_i^n + \lambda m_i (q_{i+\frac12}^n - q_{i-\frac12}^n) 
  \\ \nonumber &   \qquad 
  -\big(1 + \lambda (\beta b_{i+\frac12} + \gamma a_i^{n+1})\big) \Big( (q_{i+\frac12}^n - q_{i-\frac12}^n) (\Sigma^{n+1}_{i+\frac12})^+ + (q_{i+\frac32}^n - q_{i+\frac12}^n) (\Sigma^{n+1}_{i+\frac12})^- \Big) \rho_{i+\frac12}^{n+1}
  \\ \nonumber &   \qquad 
  + \lambda (\beta b_{i-\frac12} + \gamma a_i^{n+1}) \Big( (q_{i-\frac12}^n - q_{i-\frac32}^n) (\Sigma^{n+1}_{i-\frac12})^+ + (q_{i+\frac12}^n - q_{i-\frac12}^n) (\Sigma^{n+1}_{i-\frac12})^-  \Big) \rho_{i-\frac12}^{n+1}
  \\
  &= q_{i+\frac12}^n -\bar q_i^n  -\big(1 + \lambda (\beta b_{i+\frac12} + \gamma a_i^{n+1})\big)  (q_{i+\frac32}^n - q_{i+\frac12}^n) (\Sigma^{n+1}_{i+\frac12})^-  \rho_{i+\frac12}^{n+1}
  \\ \nonumber &\qquad 
  + \Big( \lambda m_i  -\big(1 + \lambda (\beta b_{i+\frac12} + \gamma a_i^{n+1})\big)   (\Sigma^{n+1}_{i+\frac12})^+  \rho_{i+\frac12}^{n+1}+ \lambda (\beta b_{i-\frac12} + \gamma a_i^{n+1})  (\Sigma^{n+1}_{i-\frac12})^-   \rho_{i-\frac12}^{n+1} \Big) (q_{i+\frac12}^n - q_{i-\frac12}^n) 
  \\ \nonumber &   \qquad 
  + \lambda (\beta b_{i-\frac12} + \gamma a_i^{n+1})  (q_{i-\frac12}^n - q_{i-\frac32}^n) (\Sigma^{n+1}_{i-\frac12})^+  \rho_{i-\frac12}^{n+1}
 \end{align}}
as well as
{\footnotesize
\begin{align}
 &\bar q_i^{n+1} - q_{i-\frac12}^{n+1} = \bar q_i^n - q_{i-\frac12}^n - \lambda m_i (q_{i+\frac12}^n - q_{i-\frac12}^n)
 + \Delta q_{i-\frac12}^{n+1,\text{upw}} \rho_{i-\frac12}^{n+1}   
 \\ \nonumber &\qquad
 + \lambda( \beta b_{i+\frac12}  + \gamma a_i^{n+1} )\Delta q_{i+\frac12}^{n+1,\text{upw}} \rho_{i+\frac12}^{n+1}
 - \lambda(\beta b_{i-\frac12} + \gamma a_i^{n+1}) \Delta q_{i-\frac12}^{n+1,\text{upw}} \rho_{i-\frac12}^{n+1} \\
 &= \bar q_i^n - q_{i-\frac12}^n - \lambda m_i (q_{i+\frac12}^n - q_{i-\frac12}^n)
 \\ \nonumber &\qquad
 + \lambda( \beta b_{i+\frac12}  + \gamma a_i^{n+1} )\Big ( (q_{i+\frac12}^n - q_{i-\frac12}^n) (\Sigma^{n+1}_{i+\frac12})^+ + (q_{i+\frac32}^n - q_{i+\frac12}^n) (\Sigma^{n+1}_{i+\frac12})^- \Big ) \rho_{i+\frac12}^{n+1}
 \\ \nonumber &\qquad
 + (1 - \lambda(\beta b_{i-\frac12} + \gamma a_i^{n+1})) \Big( (q_{i-\frac12}^n - q_{i-\frac32}^n) (\Sigma^{n+1}_{i-\frac12})^+ + (q_{i+\frac12}^n - q_{i-\frac12}^n) (\Sigma^{n+1}_{i-\frac12})^-  \Big) \rho_{i-\frac12}^{n+1} 
 \\
 &= \bar q_i^n - q_{i-\frac12}^n + \lambda( \beta b_{i+\frac12}  + \gamma a_i^{n+1} ) (q_{i+\frac32}^n - q_{i+\frac12}^n) (\Sigma^{n+1}_{i+\frac12})^-  \rho_{i+\frac12}^{n+1}
 \\ \nonumber &\qquad
 + \Big ( - \lambda m_i 
 + \lambda( \beta b_{i+\frac12}  + \gamma a_i^{n+1} )  (\Sigma^{n+1}_{i+\frac12})^+  \rho_{i+\frac12}^{n+1}
 + (1 - \lambda(\beta b_{i-\frac12} + \gamma a_i^{n+1}))   (\Sigma^{n+1}_{i-\frac12})^-  \rho_{i-\frac12}^{n+1} \Big) (q_{i+\frac12}^n - q_{i-\frac12}^n)
 \\ \nonumber &\qquad
  +(1 - \lambda(\beta b_{i-\frac12} + \gamma a_i^{n+1}))  (q_{i-\frac12}^n - q_{i-\frac32}^n) (\Sigma^{n+1}_{i-\frac12})^+  \rho_{i-\frac12}^{n+1}
\end{align} }

From Lemma \ref{lem:harten} one thus concludes the conditions sufficient for TV stability:
{\footnotesize
\begin{align}
  0 &\leq R_i^n + \Big( \lambda m_i  -\big(1 + \lambda (\beta b_{i+\frac12} + \gamma a_i^{n+1})\big)   (\Sigma^{n+1}_{i+\frac12})^+  \rho_{i+\frac12}^{n+1}+ \lambda (\beta b_{i-\frac12} + \gamma a_i^{n+1})  (\Sigma^{n+1}_{i-\frac12})^-   \rho_{i-\frac12}^{n+1} \Big)( R_i^n  + 1) \\
   0 &\leq -\big(1 + \lambda (\beta b_{i+\frac12} + \gamma a_i^{n+1})\big)  (\Sigma^{n+1}_{i+\frac12})^-  \rho_{i+\frac12}^{n+1} (R_{i+1}^n + 1) \label{eq:inequality2} \\
   0 &\leq \lambda (\beta b_{i-\frac12} + \gamma a_i^{n+1})   (\Sigma^{n+1}_{i-\frac12})^+  \rho_{i-\frac12}^{n+1} (R_{i-1}^n + 1) \\  
  0 &\leq 1 + \Big(- \lambda m_i 
 + \lambda( \beta b_{i+\frac12}  + \gamma a_i^{n+1} )  (\Sigma^{n+1}_{i+\frac12})^+  \rho_{i+\frac12}^{n+1}
 + (1 - \lambda(\beta b_{i-\frac12} + \gamma a_i^{n+1}))   (\Sigma^{n+1}_{i-\frac12})^-  \rho_{i-\frac12}^{n+1}\Big) (R_i^n + 1)  \\
  0 &\leq \Big(\lambda( \beta b_{i+\frac12}  + \gamma a_i^{n+1} )  (\Sigma^{n+1}_{i+\frac12})^-  \rho_{i+\frac12}^{n+1} \Big) (R_{i+1}^n + 1) \label{eq:inequality5} \\
  0 &\leq \Big( (1 - \lambda(\beta b_{i-\frac12} + \gamma a_i^{n+1}))  (\Sigma^{n+1}_{i-\frac12})^+  \rho_{i-\frac12}^{n+1} \Big) (R_{i-1}^n + 1) 
 \end{align}}
as well as \eqref{eq:consistencycondlemma}, which vanishes identically: {\footnotesize
 \begin{align}
   \Big(  -\big(1 + \lambda (\beta b_{i+\frac12} + \gamma a_i^{n+1})\big)    
   + \lambda (\beta b_{i+\frac12} + \gamma a_{i+1}^{n+1})    
 + \lambda( \beta b_{i+\frac12}  + \gamma a_i^{n+1} )  
 + (1 - \lambda(\beta b_{i+\frac12} + \gamma a_{i+1}^{n+1})) \Big) (\Sigma^{n+1}_{i+\frac12})^+  \rho_{i+\frac12}^{n+1} 
\\\nonumber 
  + \Big(  -\big(1 + \lambda (\beta b_{i-\frac12} + \gamma a_{i-1}^{n+1})\big)   
   + \lambda (\beta b_{i-\frac12} + \gamma a_i^{n+1})    
   + \lambda( \beta b_{i-\frac12}  + \gamma a_{i-1}^{n+1} )  
 + (1 - \lambda(\beta b_{i-\frac12} + \gamma a_i^{n+1})) \Big)  (\Sigma^{n+1}_{i-\frac12})^- \rho_{i-\frac12}^{n+1} 
 =  0
 \end{align}

 }

For uniformly positive speeds inequalities \eqref{eq:inequality2} and \eqref{eq:inequality5} are trivial, and the remaining ones read

\begin{align}
  0 &\leq R_i^n + \Big( \lambda m_i  -\big(1 + \lambda (\beta b_{i+\frac12} + \gamma a_i^{n+1})\big)    \rho_{i+\frac12}^{n+1} \Big)( R_i^n  + 1) \label{eq:posspeeddirectinequalities1}\\
    0 &\leq \lambda (\beta b_{i-\frac12} + \gamma a_i^{n+1})    \rho_{i-\frac12}^{n+1} (R_{i-1}^n + 1) \\  
  0 &\leq 1 + \Big(- \lambda m_i 
 + \lambda( \beta b_{i+\frac12}  + \gamma a_i^{n+1} )    \rho_{i+\frac12}^{n+1}
 \Big) (R_i^n + 1)  \\
   0 &\leq  (1 - \lambda(\beta b_{i-\frac12} + \gamma a_i^{n+1}))   \rho_{i-\frac12}^{n+1}  (R_{i-1}^n + 1) \label{eq:posspeeddirectinequalities4}
 \end{align}

which can be summarized as the following inequality for the combination $\beta b_{i+\frac12} \rho_{i+\frac12}^{n+1}$: 
{\footnotesize
\begin{align}
    & \max \left( -\gamma a_{i+1}^{n+1}\rho_{i+\frac12}^{n+1},    m_i-\frac{1}{\lambda(R_i^n + 1)}    -  \rho_{i+\frac12}^{n+1}\gamma a_i^{n+1}    \right) \label{eq:allconditionstogether} 
    \leq 
    \beta b_{i+\frac12} \rho_{i+\frac12}^{n+1}    \leq \qquad \qquad \qquad \\ \nonumber &  \qquad \qquad \qquad \qquad \qquad \qquad \min\left( m_i +\frac{R_i^n}{\lambda(R_i^n  + 1)}  - \frac{1 + \lambda \gamma a_i^{n+1} }{\lambda}    \rho_{i+\frac12}^{n+1} , \frac{1  -  \lambda\gamma a_{i+1}^{n+1}}{\lambda} \rho_{i+\frac12}^{n+1}    \right)
 \end{align}}

 \begin{theorem}
  Consider linear advection \eqref{eq:linadv} with $\cfl := c \lambda > 0$. Then Active Flux is TVD provided that the cell interface flux $\hat f_{i+\frac12}= c q_{i+\frac12}^{n} 
 - s_{i+\frac12} c(q_{i+\frac12}^{n} - q_{i-\frac12}^{n})$ is chosen such that
  \begin{align}
    s_\text{min} := \max \left( 0,    1-\frac{1}{\cfl(R_i^n + 1)}      \right)  &\leq 
    s_{i+\frac12}  \leq  \min\left( 1 +\frac{R_i^n}{\cfl(R_i^n  + 1)}  - \frac{ \rho_{i+\frac12}^{n+1}   }{\cfl}   , \frac{\rho_{i+\frac12}^{n+1}  }{\cfl}     \right)=:s_\text{max} \label{eq:ineqialitylinear}
 \end{align}
 and for monotone data this interval is non-empty if $\rho_{i+\frac12}^{n+1}$ is given by the power law \eqref{eq:powerlaw}, and otherwise by the parabola \eqref{eq:parabolarho}.
 \end{theorem}
  \begin{proof}
  One has $m_i = c$, and thus inequality \eqref{eq:allconditionstogether} reduces to
  \begin{align}
    & \max \left( 0,    1-\frac{1}{\cfl(R_i^n + 1)}      \right)  &\leq 
    \left(\beta \frac{b_{i+\frac12}}{c}  +  \gamma \right )\rho_{i+\frac12}^{n+1}  \leq  \min\left( 1 +\frac{R_i^n}{\cfl(R_i^n  + 1)}  - \frac{ \rho_{i+\frac12}^{n+1}   }{\cfl}   , \frac{\rho_{i+\frac12}^{n+1}  }{\cfl}     \right)
 \end{align}
 With $c > 0$ one finds from \eqref{eq:deltaqdef} that $ \Delta q_{i+\frac12}^{n+1,\text{upw}} :=  q_{i+\frac12}^n - q_{i-\frac12}^n$ and hence from \eqref{eq:auxflux2} that
 \begin{align}
 \frac{f^{n+\frac12}_{i+\frac12}}{c} = q_{i+\frac12}^n -\frac{ b_{i+\frac12}}{c} \rho_{i+\frac12}^{n+1} (q_{i+\frac12}^n - q_{i-\frac12}^n)
\end{align}
and, with \eqref{eq:defofrho}
\begin{align}
 \hat f_{i+\frac12}&= \alpha c q_{i+\frac12}^n + \beta f^{n+\frac12}_{i+\frac12} + \gamma c q_{i+\frac12}^{n+1} \\
 &=  c q_{i+\frac12}^{n} 
 - \underbrace{\left( \beta \frac{b_{i+\frac12}}{c} + \gamma  \right)}_{=: s_{i+\frac12}} \rho_{i+\frac12}^{n+1}c(q_{i+\frac12}^{n} - q_{i-\frac12}^{n})
\end{align}
Concerning the fact that the interval defining $s_{i+\frac12}$ is non-empty, Figure \ref{fig:choiceinequality} shows $s_\text{max} - s_\text{min}$ as a function of $R_i^n \in (0, \infty)$ and $\cfl \in (0,1]$. The detailed proof is tedious, but the Figure shows that the interesting situations only really occur for small and large $R_i^n$. For example, for $R_i^n < \frac12$ and $\nu < \frac12$ the inequality $s_\text{min} \leq s_\text{max}$ reduces to $0 \leq \cfl^{\frac1{R_i^n} - 1}$ which is always true. For large $R_i^n$ (e.g. for $R_i^n > 2$, $\cfl > \frac13$) the inequality reads
\begin{align}
 1 - \frac{1}{\cfl(R_i^n + 1)} \leq 1 + \frac{R_i^n}{\cfl(R_i^n +1)} - \frac{1 - (1-\cfl)^{R_i^n}}{\cfl}
\end{align}
which reduces to $0 \leq \frac{(1  - \cfl)^{R_i^n}}{\cfl}$, that also is always true.

\begin{figure}
 \centering
 \includegraphics[width=0.7\textwidth]{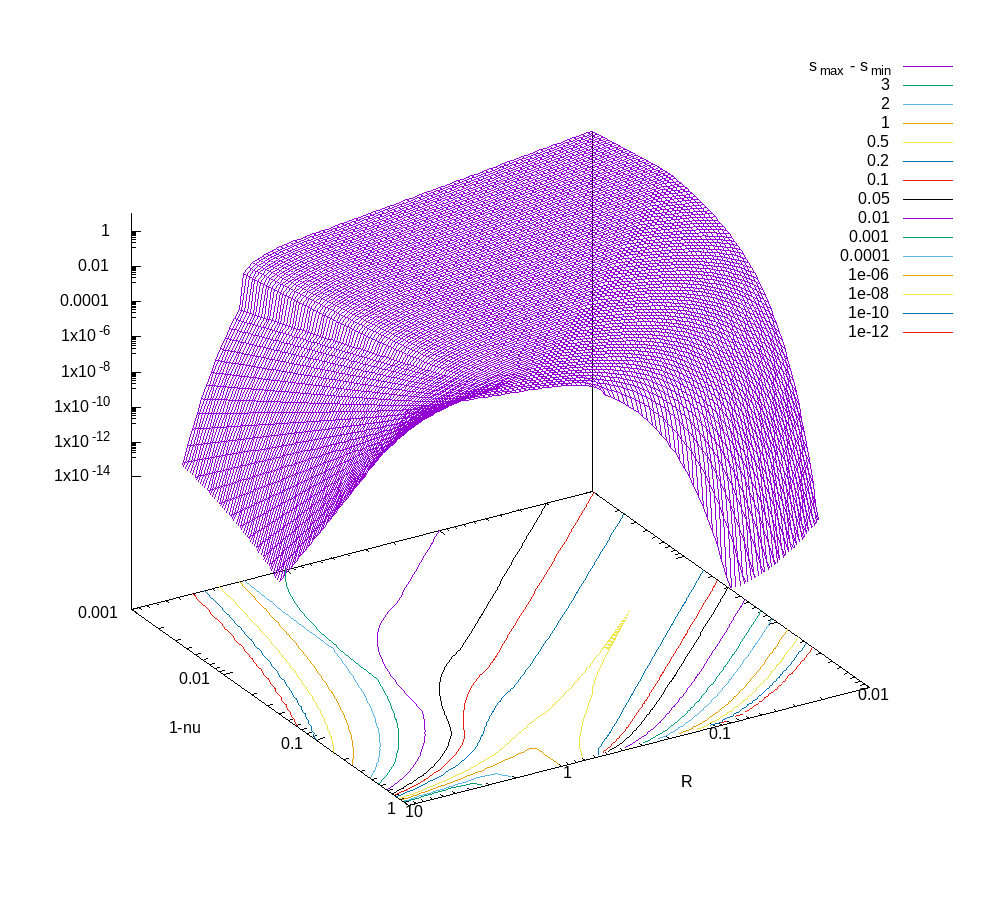}
 \caption{The difference between the upper and lower bound in \eqref{eq:ineqialitylinear} is always non-negative.}
 \label{fig:choiceinequality}
\end{figure}
 \end{proof}

For linear advection, limiting the flux at half the time step therefore is really rather simple: If one finds the default value obtained using the original Active Flux algorithm (applied to a monotonized reconstruction) to be outside the bounds of $\eqref{eq:ineqialitylinear}$, one simply can take any value inside these bounds instead.
For nonlinear equations the situation is slightly more complicated and in certain cases one needs to modify the new point value $q_{i+\frac12}^{n+1}$, or equivalently, $\rho_{i+\frac12}^{n+1}$ in order for the interval in \eqref{eq:allconditionstogether} to be non-empty, as is discussed next.

\begin{theorem} \label{thm:sufficient}
 Assume uniformly positive speeds in \eqref{eq:allconditionstogether} and assume that there exist local bounds $a_{i,\text{max}}, a_{i,\text{min}}$
  that ensure $a_{i,\text{min}} \leq   a_{i+1}^{n+1}, a_i^n, a_i^{n+1}  \leq a_{i,\text{max}}$. Assume also the CFL condition $\lambda a_{i,\text{max}} \leq \frac{\sqrt{97}-7}{4} \simeq 0.712$ (which is the positive root of $2x^2 + 7x - 6$), take $\gamma = \frac16$ and $R_i^n \geq 0$.
 Then it is sufficient for $\rho_{i+\frac12}^{n+1}$ to fulfill
 \begin{align}
  \max\left(0, \frac{\lambda a_{i,\text{max}} -\frac{1}{R_i^n + 1}}{1 - \lambda \gamma a_{i,\text{max}} }   \right)   \leq    \rho_{i+\frac12}^{n+1}  \leq 
   \min\left(1,  \frac{  \lambda a_{i,\text{min}} + \frac{R_i^n}{R_i^n  + 1} }{1 + \lambda \gamma a_{i,\text{max}} }\right)  \label{eq:suffconditionrhonp1}
 \end{align}
 in order for the interval in \eqref{eq:allconditionstogether} to be non-empty, and the interval in \eqref{eq:suffconditionrhonp1} is non-empty itself.
\end{theorem}

\begin{proof}
Use first the standard CFL condition $\lambda a_{i,\text{max}} \leq 1$ only.
The conditions for the interval in \eqref{eq:allconditionstogether} to be non-empty are
\begin{align}
 -\gamma a_{i+1}^{n+1}\rho_{i+\frac12}^{n+1} &\leq \left( \frac{1}{\lambda}  -  \gamma a_{i+1}^{n+1}\right) \rho_{i+\frac12}^{n+1} \\
 m_i-\frac{1}{\lambda(R_i^n + 1)}    -  \rho_{i+\frac12}^{n+1}\gamma a_i^{n+1}   &\leq m_i +\frac{R_i^n}{\lambda(R_i^n  + 1)}  - \frac{1 + \lambda \gamma a_i^{n+1} }{\lambda}    \rho_{i+\frac12}^{n+1}
\end{align}
which reduce to
\begin{align}
 0 &\leq \rho_{i+\frac12}^{n+1} & 
   0  &\leq  1 - \rho_{i+\frac12}^{n+1} \label{eq:suffcondineq1}
\end{align}
and also
\begin{align}
  \lambda m_i-\frac{1}{R_i^n + 1}      &\leq   \left(1 - \lambda \gamma( a_{i+1}^{n+1} - a_i^{n+1} )  \right) \rho_{i+\frac12}^{n+1}  \leq 
     \lambda m_i + \frac{R_i^n}{R_i^n  + 1}   \label{eq:suffcondineq2}
 \end{align}
 
Observe first that $a_{i+1}^{n+1} - a_i^{n+1} \leq a_{i,\text{max}}-a_{i,\text{min}}$
 and thus
 \begin{align}
  1 - \lambda \gamma( a_{i+1}^{n+1} - a_i^{n+1} ) \geq  1 - \lambda \gamma( a_{i,\text{max}}-a_{i,\text{min}} ) \geq 1 - \lambda \gamma a_{i,\text{max}} > 0
 \end{align}
 and recall that $m_i = (\alpha + \beta)a_i^n  +  \gamma a_i^{n+1}$.
  Inequalities \eqref{eq:suffcondineq1}--\eqref{eq:suffcondineq2} define a common interval for $\rho_{i+\frac12}^{n+1}$ which is non-empty, because (observe that the expression is deliberately \emph{not} $\lambda m_i$)
 \begin{align}
 \lambda (\alpha + \beta)a_i^n +\lambda \gamma a_{i+1}^{n+1} \leq \lambda a_{i,\text{max}} \leq 1 < \frac{R_i^n + 2}{R_i^n + 1}
 \end{align}
 which implies 
 \begin{align}
 \lambda (\alpha + \beta)a_i^n  -\frac{1}{R_i^n + 1}  &\leq 1 - \lambda \gamma a_{i+1}^{n+1}  \\ 
 \lambda m_i-\frac{1}{R_i^n + 1}  &\leq 1 - \lambda \gamma( a_{i+1}^{n+1} - a_i^{n+1} ) 
 \end{align}
They could in principle be used in practice. It might be cumbersome, though, to solve them for $\rho_{i+\frac12}^{n+1}$, since $a_{i+1}^{n+1}$ and $a_i^{n+1}$ are involved which themselves depend on it.

It is then sufficient to impose 
\begin{align}
  \max\left(0, \frac{\lambda a_{i,\text{max}} -\frac{1}{R_i^n + 1}}{1 - \lambda \gamma (a_{i,\text{max}}-a_{i,\text{min}}) }   \right)   &\leq    \rho_{i+\frac12}^{n+1}  \leq 
    \min\left(1, \frac{ \lambda a_{i,\text{min}} + \frac{R_i^n}{R_i^n  + 1} }{1 + \lambda \gamma (a_{i,\text{max}}-a_{i,\text{min}}) }\right)
 \end{align}
because
\begin{align}
 \max\left(0, \frac{\lambda m_i-\frac{1}{R_i^n + 1}}{1 - \lambda \gamma( a_{i+1}^{n+1} - a_i^{n+1} )}\right) \leq \frac{\max\left(0, \lambda a_{i,\text{max}} -\frac{1}{R_i^n + 1}\right)}{1 - \lambda \gamma( a_{i+1}^{n+1} - a_i^{n+1} )} \leq \frac{\max\left(0, \lambda a_{i,\text{max}} -\frac{1}{R_i^n + 1}\right)}{1 - \lambda \gamma( a_{i,\text{max}}-a_{i,\text{min}})}
\end{align}
and analogously for the other bound, using that
\begin{align}
 1 - \lambda \gamma( a_{i+1}^{n+1} - a_i^{n+1} ) \leq 1 - \lambda \gamma( a_{i,\text{min}} -a_{i,\text{max}})
\end{align}

Since $a_{i,\text{max}}\geq a_{i,\text{min}} \geq 0$, one can also simply impose \eqref{eq:suffconditionrhonp1}.
 It is at this point that the slightly reduced CFL condition is required in order to guarantee the non-emptiness of the interval in \eqref{eq:suffconditionrhonp1}. Abbreviate $\lambda a_{i,\text{max}}  =: N$. First,
 \begin{align}
  \frac{N-\frac{1}{R_i^n + 1}}{1 - N \gamma} \leq 1
 \end{align}
is equivalent to
 \begin{align}
  N \leq \frac{R_i^n + 2}{(R_i^n + 1)(1 + \gamma)}
  \end{align}
  which is guaranteed by taking $N \leq \frac{1}{1 + \gamma} = \frac{6}{7} \simeq 0.857$. Second,
 \begin{align}
  \frac{N -\frac{1}{R_i^n + 1}}{1 -  \gamma N }   \leq    \frac{  \lambda a_{i,\text{min}} + \frac{R_i^n}{R_i^n  + 1} }{1 + \gamma N }
 \end{align}
is equivalent to
 \begin{align}
   0 \leq   \lambda a_{i,\text{min}} +1 -N - \gamma N \left( \lambda a_{i,\text{min}}+ \frac{R_i^n-1}{R_i^n  + 1} + N \right)
 \end{align}
Now, since $\lambda a_{i,\text{min}} + \frac{R_i^n-1}{R_i^n  + 1} + N \leq 2N +1$,
\begin{align}
 \lambda a_{i,\text{min}} +1 -N - \gamma N \left( \lambda a_{i,\text{min}} +\frac{R_i^n-1}{R_i^n  + 1} + N \right) &\geq 
 1 - N(1 + \gamma) - 2\gamma  N^2 \geq 0
\end{align}
by assumption.

\end{proof}

Observe again that the restriction in the CFL condition arises only because the actual wave speeds have been replaced by estimates, and that the classical condition $\lambda a_{i,\text{max}} \leq 1$ could be retained at the expense of having to deal with a more complicated inequality \eqref{eq:suffcondineq1}--\eqref{eq:suffcondineq2}. It is possible that the CFL estimate can be improved by refining the sufficient conditions.

 For reference, the corresponding inequalities for uniformly negative speeds are 
 \begin{align}
 m_{i+1} +\frac{R_{i+1}^n}{\lambda(R_{i+1}^n  + 1)} - \rho_{i+\frac12}^{n+1} \gamma a_{i+1}^{n+1}  &\geq   \beta b_{i+\frac12}  \rho_{i+\frac12}^{n+1} \\
   -\left( \frac{1}{\lambda}+\gamma a_i^{n+1}\right)\rho_{i+\frac12}^{n+1} &\leq  \beta b_{i+\frac12}\rho_{i+\frac12}^{n+1}   \\
   m_{i+1} - \frac{1}{\lambda(R_{i+1}^n + 1)} +\frac{ \rho_{i+\frac12}^{n+1}}{\lambda } -\gamma a_{i+1}^{n+1}\rho_{i+\frac12}^{n+1} &\leq    
  \beta b_{i+\frac12}\rho_{i+\frac12}^{n+1}       \\
  - \gamma a_i^{n+1} \rho_{i+\frac12}^{n+1} &\geq  \beta b_{i+\frac12} \rho_{i+\frac12}^{n+1}
 \end{align}
 i.e. {\footnotesize
 \begin{align}
   &\max\left(
   -\left( \frac{1}{\lambda}+\gamma a_i^{n+1}\right)\rho_{i+\frac12}^{n+1},   
   m_{i+1} - \frac{1}{\lambda(R_{i+1}^n + 1)} +\frac{ \rho_{i+\frac12}^{n+1}}{\lambda } -\gamma a_{i+1}^{n+1}\rho_{i+\frac12}^{n+1}    
   \right ) \\ \nonumber &\qquad \qquad \qquad  \leq \beta b_{i+\frac12} \rho_{i+\frac12}^{n+1} \leq  \\ \nonumber & \qquad \qquad \qquad\qquad \qquad \qquad \min\left(
   - \gamma a_i^{n+1} \rho_{i+\frac12}^{n+1},   
   m_{i+1} +\frac{R_{i+1}^n}{\lambda(R_{i+1}^n  + 1)} - \rho_{i+\frac12}^{n+1} \gamma a_{i+1}^{n+1}  
   \right)
 \end{align}
 }

 Again, the interval is non-empty upon the following choice of $\rho_{i+\frac12}^{n+1}$:
 \begin{align}
   \rho_{i+\frac12}^{n+1} &\geq0 & \rho_{i+\frac12}^{n+1}  & \leq   1
  \end{align}

 \begin{align}
   -\lambda m_{i+1} -\frac{R_{i+1}^n}{R_{i+1}^n  + 1}   \leq  \rho_{i+\frac12}^{n+1}\left( 1 -\gamma \lambda( a_{i+1}^{n+1} - a_i^{n+1}) \right) & \leq - \lambda m_{i+1} +\frac{1}{R_{i+1}^n + 1}  
 \end{align}

 As before, it is sufficient to choose $\rho_{i+\frac12}^{n+1}$ such that 
 \begin{align}
  \max \left(0, -\frac{\lambda a_{i,\text{min}} +\frac{R_{i+1}^n}{R_{i+1}^n  + 1} }{1 +\gamma \lambda a_{i,\text{min}}}\right)  \leq  \rho_{i+\frac12}^{n+1} & \leq \min\left(1, \frac{- \lambda a_{i,\text{max}} +\frac{1}{R_{i+1}^n + 1}  }{1 -\gamma \lambda a_{i,\text{min}}}\right)
 \end{align}

 The algorithm of the TVD limiting of Active Flux can now be summarized as follows:
 \begin{enumerate}
  \item Compute all new point values $\rho_{i+\frac12}^{n+1}$ according to \eqref{eq:updatepoints} using a monotonized reconstruction.
  \item Compute all half-time-step flux evaluations $f_{i+\frac12}^{n+\frac12}$ as $f(\rho_{i+\frac12}^{n+\frac12})$ according to \eqref{eq:updatepoints} analogously.
  \item Check condition \eqref{eq:suffconditionrhonp1} and modify $\rho_{i+\frac12}^{n+1}$ if necessary.
  \item Compute $q_{i+\frac12}^{n+1}$ from $\rho_{i+\frac12}^{n+1}$ using \eqref{eq:defofrho}. This is the new point value.
  \item Compute $b_{i+\frac12}$ from \eqref{eq:auxflux}.
  \item Verify that \eqref{eq:allconditionstogether} is fulfilled, and otherwise modify $b_{i+\frac12}$ and hence $f_{i+\frac12}^{n+\frac12}$.
  \item Compute the intercell flux $\hat f_{i+\frac12}$ from \eqref{eq:intercellfluxquadrature}, and compute the new cell average $\bar q_i^{n+1}$ using \eqref{eq:updateaverage}.
 \end{enumerate}

 Limiting is applied only if $R_i^n > 0$, because otherwise the cell contains a discrete extremum. At a such, there are competing interests between continuity and monotonicity. On the one hand, for a TVD method one needs to cap the extremum, i.e. to define $q^n_{\text{recon},i} = \bar q_i^n$ in such a cell and to take (e.g. for positive speeds)
 \begin{align}
  q_{i+\frac12}^{n+1} &= \bar q_i^n & \hat f_{i+\frac12} &= f(\bar q_i^n)
 \end{align}
 Observe that continuity of the reconstruction is lost. On the other hand, one might content oneself with a method that is TVD away from local extrema and simply let standard Active Flux evolve a cell containing a discrete extremum. The latter approach is used in the numerical examples below.
 
 The remainder of this Section goes back to the examples of Section \ref{sec:examples} and shows that the prescription \eqref{eq:ineqialitylinear} indeed gives the correct fluxes in these cases.

\begin{example}
 Assume $R_i^n$ large and thus $\rho_i^{n+1} = 1$. Then, 
\begin{align*}
   &\max\left(0, 1-\frac{1}{\cfl(R_i^n+1)} \right) = 1-\frac{1}{\cfl(R_i^n+1)} \leq s_{i+\frac12}    \\\nonumber &\qquad \qquad \qquad \leq 1- \frac{1}{\cfl(R_i^n+1)} = \min \left(\frac{1 }{\cfl},1- \frac{1}{\cfl(R_i^n+1)}\right)
 \end{align*}
 and thus
 \begin{align}
 \hat f_{i+\frac12}/c &= q_{i+\frac12}^n -  ( q_{i+\frac12}^n  -  q_{i-\frac12}^n)s_{i+\frac12} = q_{i+\frac12}^n -  ( q_{i+\frac12}^n  -  q_{i-\frac12}^n) \left( 1- \frac{\bar q_i^n - q_{i-\frac12}^n}{\cfl( q_{i+\frac12}^n  - q_{i-\frac12}^n)}\right)\\
 &=   q_{i-\frac12}^n + \frac{\bar q_i^n - q_{i-\frac12}^n}{\cfl}\\
 \intertext{or, in the notation of the example of Section \ref{ssec:examplesrightmovjump}}
 &=  q_{i-\frac12}^n + \frac{\epsilon}{\cfl}
\end{align}
which agrees with the result obtained there.
\end{example}

 \begin{example}
  Consider again linear advection \eqref{eq:linadv} with $c > 0$. If $\cfl = 1$, one finds from \eqref{eq:defofrho} and \eqref{eq:deltaqdef} $\rho_i^{n+1} = 1$ and
 \begin{align}
   \max \left(0, 1-\frac{1}{R_i^n+1}\right) \leq  s_{i+\frac12}  \leq \min\left( 1, \frac{R_i^n}{R_i^n+1} \right) 
 \end{align}
i.e.
\begin{align}
 \hat f_{i+\frac12}/c 
   &= q_{i+\frac12}^n -  ( q_{i+\frac12}^n  -  q_{i-\frac12}^n) \frac{R_i^n}{R_i^n + 1} = q_{i+\frac12}^n -  ( q_{i+\frac12}^n  -  q_{i-\frac12}^n) \frac{q_{i+\frac12}^n - \bar q_i^n}{q_{i+\frac12}^n - q_{i-\frac12}^n} = \bar q_i^n
\end{align}
This is the correct result, as discussed in Section \ref{ssec:examplesnuone}.
 \end{example}

 \section{Numerical examples} \label{sec:numerical}
 
 In the numerical examples below, periodic boundary conditions for a domain $[0,1]$ are used throughout. The CFL number is computed using both the point values and the averages.
 
 When monotone initial data are converted into initial cell averages and point values, it is important not to introduce an artificial oscillation at this step already, if one uses a quadrature to compute the initial cell averages. Either one should then compute them exactly, or using a quadrature that is guaranteed to yield a monotone result. Below, the trapezoidal rule is used as a fallback if Simpson's rule does not produce monotone discrete data in some cell. As discussed above, the method (denoted here by TVD Active Flux) is TVD away from local extrema.
 
 The definition of $R_i^n$ and some other parameters requires divisions by differences between values of the degrees of freedom. The limiting is not applied if these differences become smaller than $10^{-10}$. The power laws are replaced by parabolae if the exponent is larger than 50 or smaller than $\frac1{50}$ to avoid machine precision issues.
 
 \subsection{Linear advection}

 For linear advection \eqref{eq:linadv} the speed is chosen as $c=1$ and a CFL condition of 0.9 is adopted.
 
 \subsubsection{Smooth and discontinuous profiles} \label{ssec:zalesak}
 
 This is a variant of the test from \cite{jiang96}, which combines different profiles of various degree of regularity. In Figure \ref{fig:adv-zalesak} results are shown on a grid of 200 cells at $t=1$, i.e. after the entire datum has traveled once around the domain. With only a monotonized reconstruction not all oscillations are removed (the solution tends to be better for $\cfl < 0.5$, though). The TVD Active Flux solution is entirely free of oscillations. At $t=1$, the values of the numerical solution are contained in the interval $[-2.1 \cdot 10^{-10}, 1 + 2.1\cdot 10^{-10}]$. The smooth maximum of the Gaussian is visibly reduced by both versions of the limiting, while the sharp maximum of the hat function suffers from smoothing already for the unlimited method. From the same Figure one can see that the flux limiting (i.e. the replacement of $f(q_{i+\frac12}^{n+\frac12})$ by $f_{i+\frac12}^{n+\frac12}$ defined in \eqref{eq:auxflux}) happens only at locations of irregularity of the solution, and occasionally close to maxima; smooth regions are entirely unaffected.

 \begin{figure}
   \centering
   \includegraphics[width=0.9\textwidth]{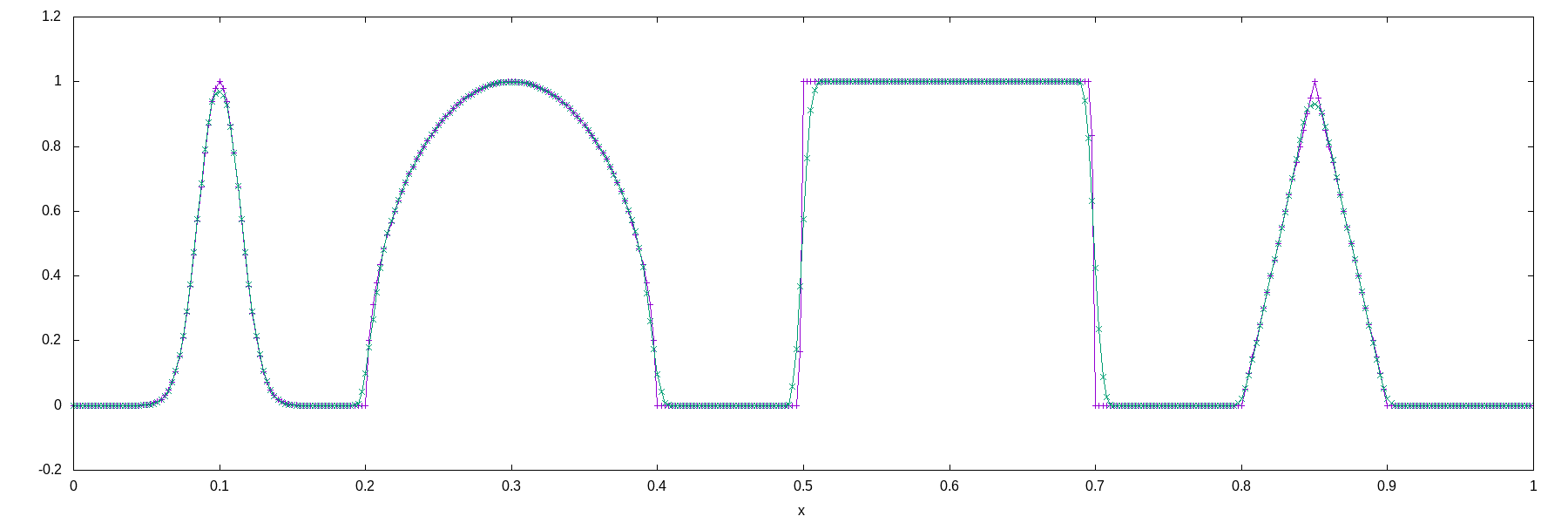} \\
   \includegraphics[width=0.9\textwidth]{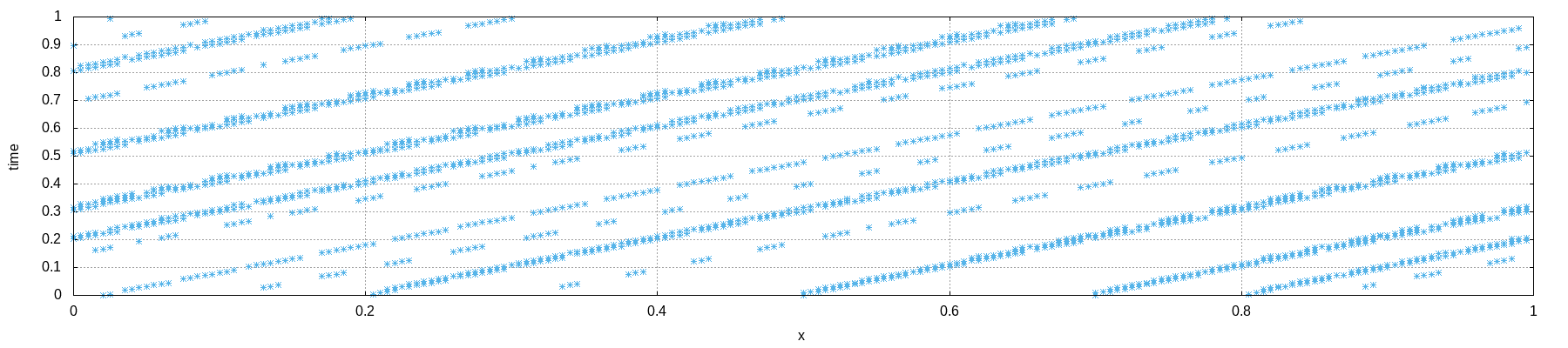}\\
   \includegraphics[width=0.9\textwidth]{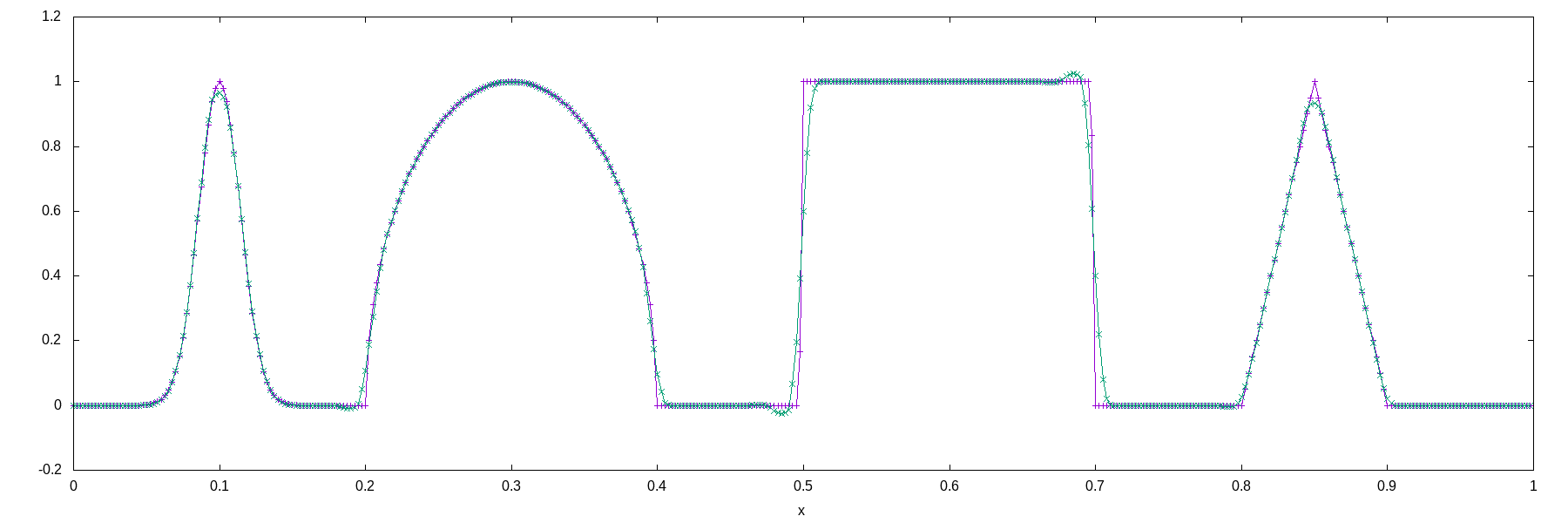} \\
   \includegraphics[width=0.9\textwidth]{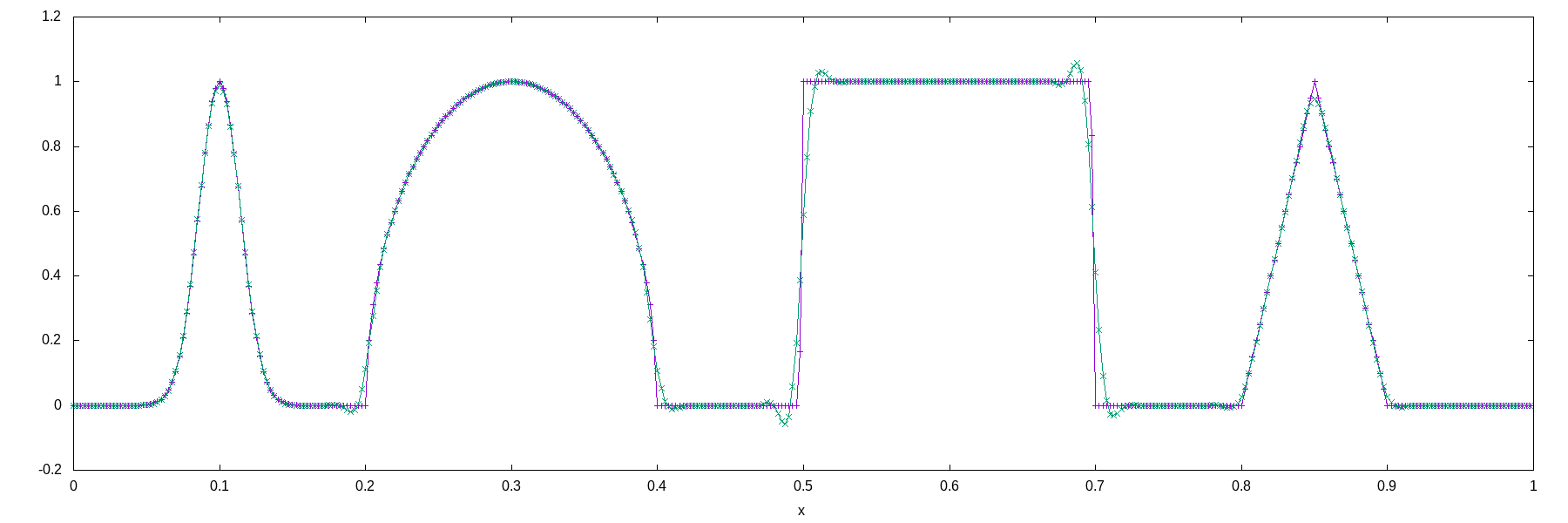}
   \caption{Linear advection with initial data containing smooth and discontinuous profiles (Section \ref{ssec:zalesak}). Initial data and results at time $t=1$ are shown, after the data have traveled once around the domain. \emph{From top to bottom}: Results obtained using the TVD Active Flux method; space-time locations of limiter activation, with crosses indicating modifications of the half-time-step flux $f_{i+\frac12}^{n+\frac12}$; results obtained using only a monotonized reconstruction (parabola/power law), but no flux limiting; results with no limiting applied at all.}
   \label{fig:adv-zalesak}   
  \end{figure}
 
 \subsubsection{Convergence study}
 
 Smooth solutions, away from local extrema, are not subject to modifications and retain third-order accuracy of unmodified Active Flux. The convergence studies of Figures \ref{fig:adv-gaussian-conv} and \ref{fig:adv-step-conv} focus therefore on the influence of the limiting on extrema and on non-smooth solutions, respectively. Figure \ref{fig:adv-gaussian-conv} shows the $L^1$ error at time $t=0.1$ for Gaussian initial data $q_0(x) = \exp\left(-\frac{(x-x_\text{c})^2}{w^2}\right)$ with $w = 0.1$, $x_\text{c} = 0.5$. The TVD Active Flux still has a convergence order of about 2.4. Figure \ref{fig:adv-step-conv} shows the same for step initial data. Here one observes that both the TVD Active Flux and the standard unlimited Active Flux converge with a rate of about $\frac34$.

 \begin{figure}
  \centering
  \includegraphics[width=0.45\textwidth]{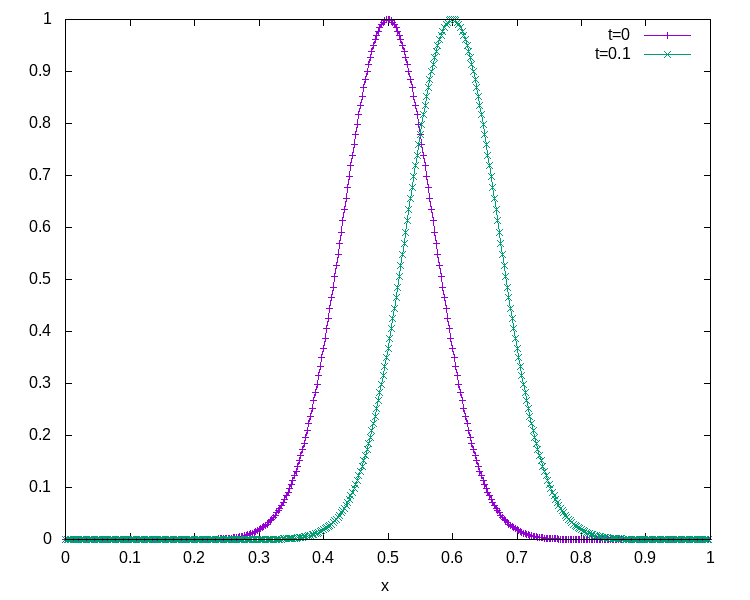} \hfill \includegraphics[width=0.45\textwidth]{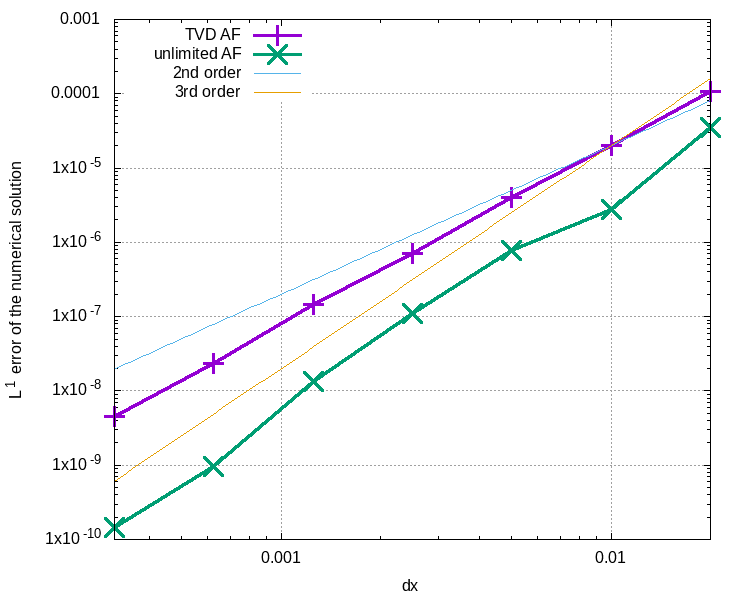}
  \caption{Convergence analysis for linear advection with Gaussian initial data on a mesh of 200 cells. \emph{Left}: Initial data and solution at time $t=0.1$ obtained using TVD Active Flux. The values at time $t=0.1$ are contained in the interval $ [1.7\cdot 10^{-11}, 0.99965]$. \emph{Right}: $L^1$ error of the numerical solution (point values) at $t=0.1$ as a function of $\Delta x$. One observes a convergence rate of about 2.4 for the TVD Active Flux, and order 3 for the unlimited Active Flux.}
  \label{fig:adv-gaussian-conv}
 \end{figure}

 \begin{figure}
  \centering
  \includegraphics[width=0.45\textwidth]{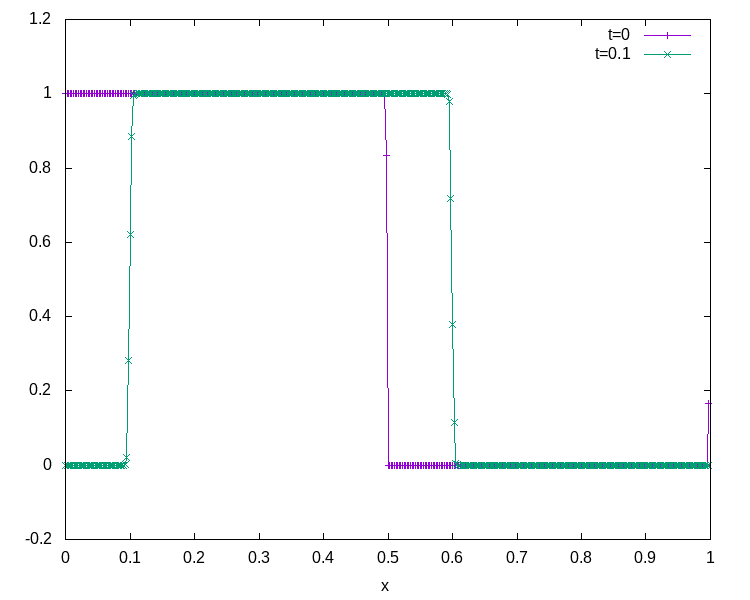} \hfill \includegraphics[width=0.45\textwidth]{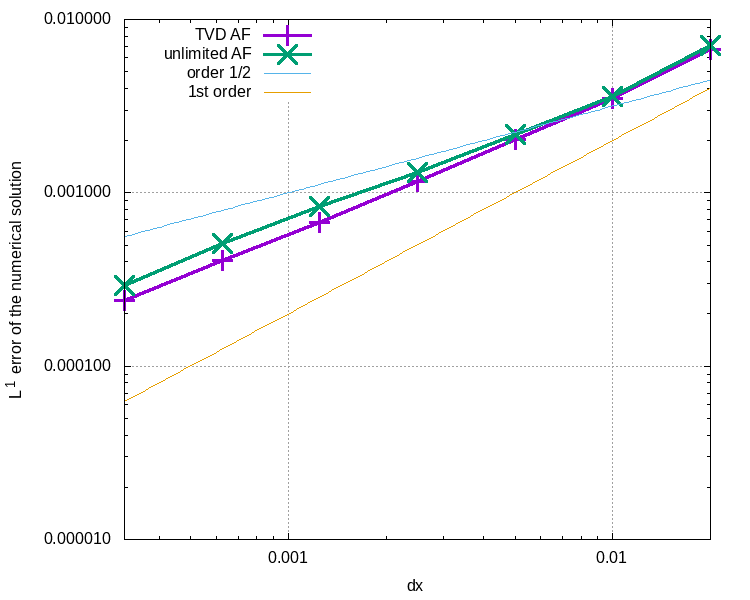}
  \caption{Convergence analysis for linear advection with step initial data on a mesh of 200 cells. \emph{Left}: Solution at times $t=0$ and $t=0.1$ obtained using TVD Active Flux. The values of the numerical solution are contained in the interval $[-2.6 \cdot 10^{-11}, 1 + 2.6 \cdot 10^{-11}]$. \emph{Right}: $L^1$ error of the numerical solution (point values) at $t=0.1$ as a function of $\Delta x$, computed in the inner 60\% of the domain in order to focus entirely on the single step. One observes for both the unlimited and the TVD Active Flux methods a convergence rate of about 0.75.}
  \label{fig:adv-step-conv}
 \end{figure}

 \subsection{Burgers' equation}
 For Burgers' equation, a CFL condition of 0.7 is adopted in accordance with Theorem \ref{thm:sufficient}, and the sufficient condition \eqref{eq:suffconditionrhonp1} is used throughout.

 \subsubsection{Hat data}
 
 Figure \ref{fig:burgers-linear} shows the evolution of hat data
 \begin{align}
   0.1 + \max\left(0, 1.1 - \left|\frac{x}{w}\right|\right)
 \end{align}
 on a grid of 200 cells with $w = 0.25$. The shock forms at $t=0.25$; one finds that at $t = 0.1$ the values of the solution are contained in $[0.1   - 7\cdot 10^{-14}, 1.0902]$, and in $[0.1 - 2 \cdot 10^{-13}, 1.0884]$ at time $t=0.2$. One observes the activation of the limiter on the shock path and close to the maximum of the hat. At early times, it also is activated close to the right corner of the hat. The solution obtained with TVD Active Flux is free of oscillations, as expected. The solution without flux-limiting (shown in the lower half of Figure \ref{fig:burgers-linear}) demonstrates the spurious oscillations that otherwise occur on the shock, and which the monotonization of the reconstruction cannot prevent. It occurs very rarely that $\rho_{i+\frac12}^{n+1}$ needs to be modified according to \eqref{eq:suffconditionrhonp1}.

\begin{figure}
   \centering
   \includegraphics[width=0.45\textwidth]{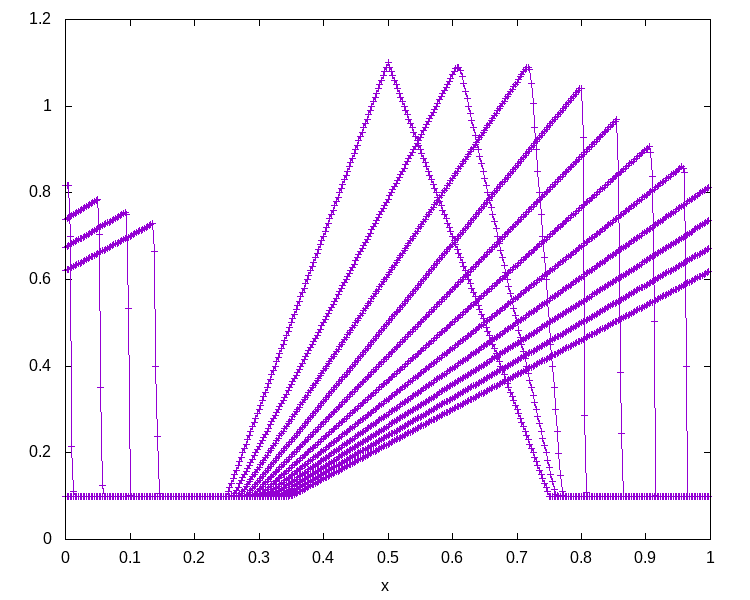} \hfill \includegraphics[width=0.45\textwidth]{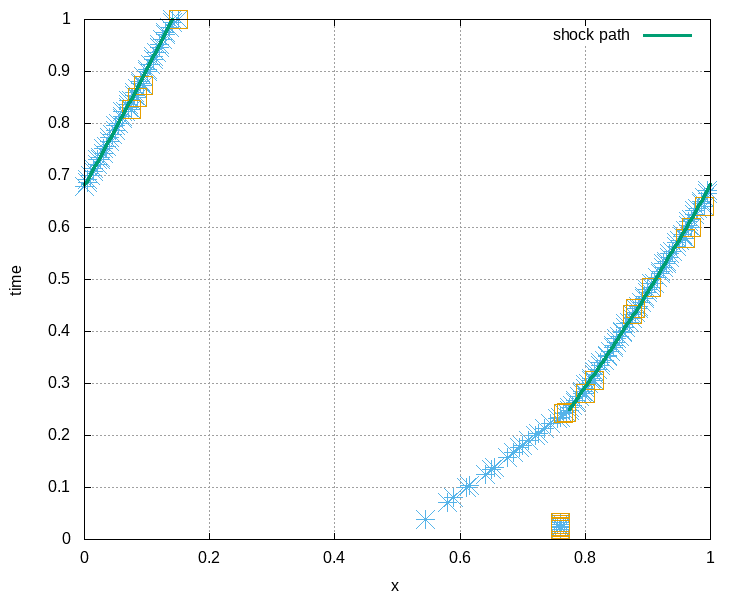}\\
   \includegraphics[width=0.45\textwidth]{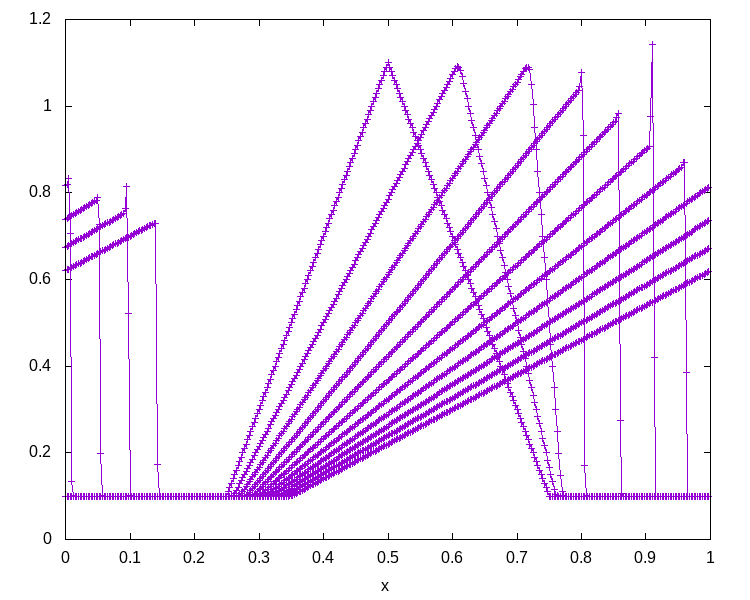} \hfill \includegraphics[width=0.45\textwidth]{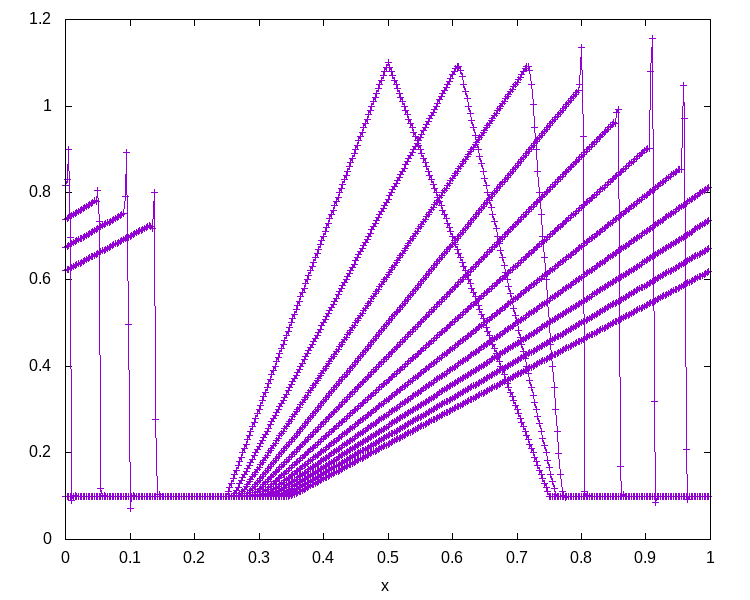}
   \caption{Burgers' equation with hat initial data. \emph{Top left}: Results obtained using the TVD Active Flux method, at times $t = 0, 0.1, 0.2, \dots$. \emph{Top right}: Space time locations of limiting activation: crosses indicate modifications of the half-time-step flux $f_{i+\frac12}^{n+\frac12}$, boxes indicate modifications of $\rho_{i+\frac12}^{n+1}$. The line indicates the analytical location of the shock after its formation. \emph{Bottom left}: Results obtained using only power-law reconstruction, but no flux limiting. \emph{Bottom right}: No limiting applied at all.}
   \label{fig:burgers-linear}   
  \end{figure}

  \begin{figure}
   \centering
   \includegraphics[width=0.45\textwidth]{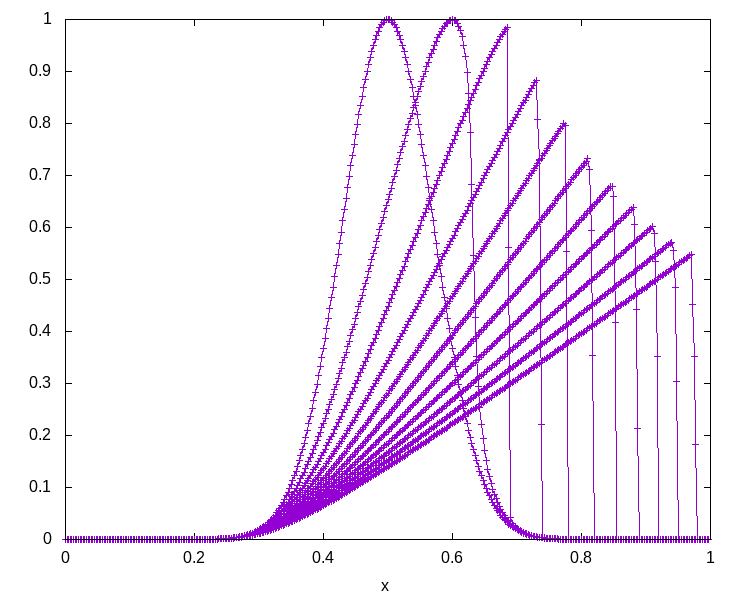} \hfill \includegraphics[width=0.45\textwidth]{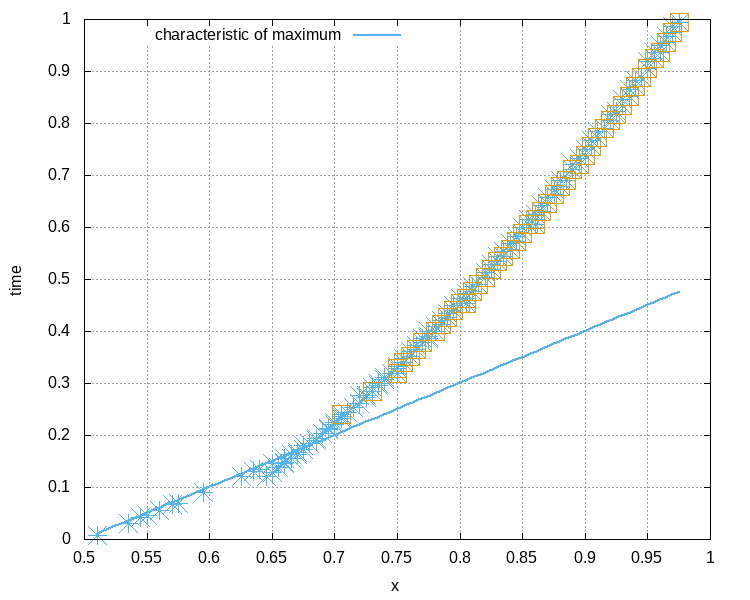}\\
   \includegraphics[width=0.45\textwidth]{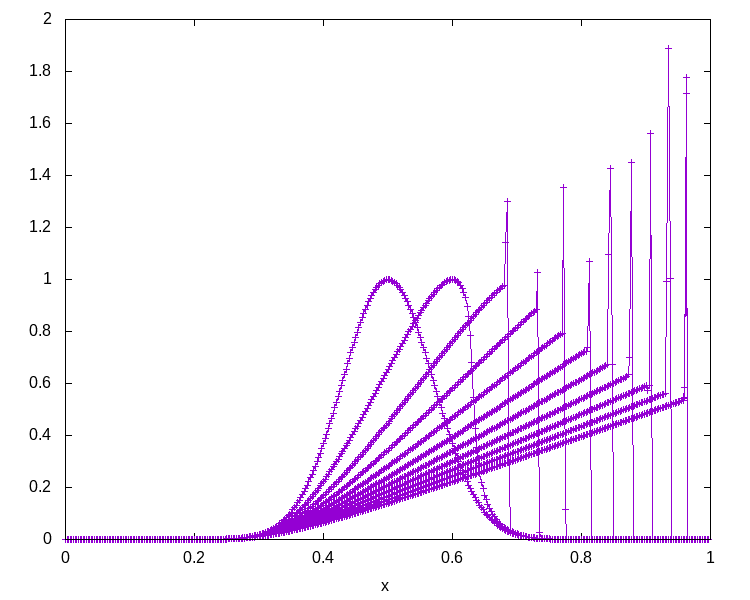} \hfill \includegraphics[width=0.45\textwidth]{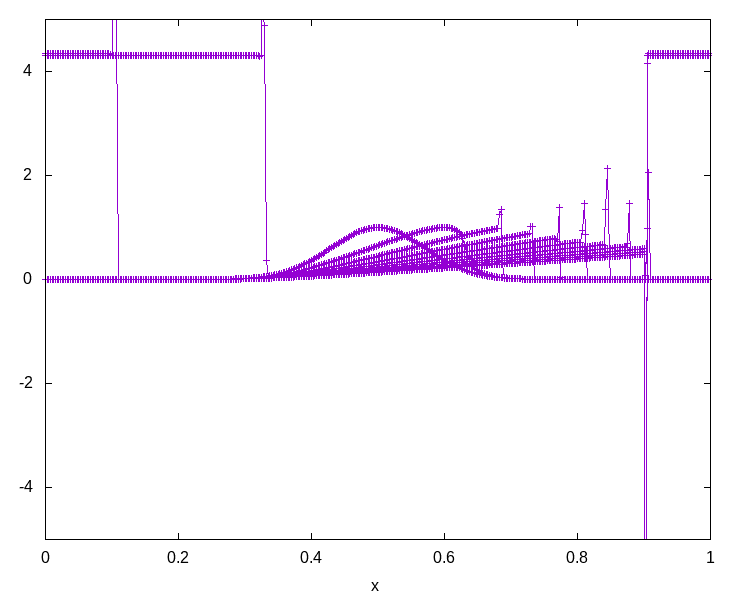}
   \caption{Burgers' equation with Gaussian initial data. \emph{Top left}: Results obtained using the TVD Active Flux method, at times $t = 0, 0.1, 0.2, \dots$. \emph{Top right}: Space time locations of limiting activation: crosses indicate modifications of the half-time-step flux $f_{i+\frac12}^{n+\frac12}$, boxes indicate modifications of $\rho_{i+\frac12}^{n+1}$. The straight line indicates the characteristic with speed 1 emanating from the grid midpoint, which the maximum of the data follows until it meets the shock. The limiter is not activated anywhere in the region not shown. \emph{Bottom left}: Results obtained using only power-law reconstruction, but no flux limiting. \emph{Bottom right}: No limiting applied at all, the solution grows in an entirely unreasonable way (oscillations outside axis range).}
   \label{fig:burgers-gaussian}   
  \end{figure}

 \subsubsection{Gaussian initial data}

  Figure \ref{fig:burgers-gaussian} shows a test case very similar to the preceding one, with the initial data being smooth. Gaussian initial data $q_0(x) = \exp\left(-\frac{(x-x_\text{c})^2}{w^2}\right)$ with $w = 0.1$, $x_\text{c} = 0.5$ self-steepen and develop a shock. A grid of 200 cells is used. Simulation results without any limiting demonstrate how the method's violent oscillations (outside the axis range shown) lead to an entirely unrecognizable solution. This is prevented by monotonized reconstruction, but the oscillation at the shock remains. It is finally removed by TVD Active Flux. The shocks forms only after $t=0.1$ and until then, the maximum is purely transported. When using the TVD Active Flux, the global maximum of the discrete data is 0.9994648896207335 at $t=0.1$. The activation of the limiter on the maximum of the Gaussian is visible in Figure \ref{fig:burgers-gaussian}, top right, too, but the frequency of its activation is significantly less than on the shock.

  \section{Conclusions and outlook}
  
  The concept of a TVD method is restricted to scalar conservation laws and/or one spatial dimension, but the theory associated with TVD methods is nevertheless of fundamental importance in numerical analysis.
  This paper presents the first TVD Active Flux method, and it does not give up continuity. Improved (monotonized) reconstructions developed a while ago are complemented here by a modified computation of the numerical flux required for the update of the cell average. This modification is based on a version of Harten's lemma that takes into account all the degrees of freedom of Active Flux. For linear advection one obtains a very simple prescription how the intercell flux needs to be chosen to guarantee the TVD property. For nonlinear scalar conservation laws one might need to first modify the point value at the new time $\rho_{i+\frac12}^{n+1}$ before being able to perform the flux modification. Even then, the limiting remains simple to implement. 
 
  Some aspects of the presented approach to TVD might be improved in future. This concerns relaxing the sufficient conditions \eqref{eq:suffconditionrhonp1} (and the CFL condition they entail), possibly exploring in more detail the regime $-1 \leq R_i^n < 0$ which seems to have some relevance for the conditions \eqref{eq:posspeeddirectinequalities1}--\eqref{eq:posspeeddirectinequalities4}, and an even better way to preserve the order of accuracy on local extrema. This also concerns transonic discontinuities.
   It is possible that similar ideas can be applied to the semi-discrete Active Flux (as introduced in \cite{abgrall20,abgrall22}) in conjunction with SSP time stepping, and that they might be useful (even if ad-hoc) for nonlinear systems such as the Euler equations. One might even hope that the TVD concept might be extended along similar lines to general continuous Finite Element approximations.
 
\vspace{1cm}

The ideas of the paper are entirely the author's and the paper has been written without AI support.
 
 \newpage
 
\bibliographystyle{alpha}

\end{document}